\pdfoutput=1
\documentclass[british,comma,draft,11pt]{article}
\usepackage[T1]{fontenc}
\usepackage[latin9]{inputenc}
\usepackage[a4paper]{geometry}
\usepackage{color}
\usepackage{babel}
\usepackage{verbatim}
\usepackage{prettyref}
\usepackage{textcomp}
\usepackage{amsthm}
\usepackage{amsmath}
\usepackage{amssymb}
\usepackage[authoryear]{natbib}
\PassOptionsToPackage{normalem}{ulem}
\usepackage{ulem}
\usepackage[unicode=true,pdfusetitle,
 bookmarks=true,bookmarksnumbered=false,bookmarksopen=false,
 breaklinks=false,pdfborder={0 0 1},backref=false,colorlinks=false]
 {hyperref}
\hypersetup{
 draft=false}

\makeatletter
\numberwithin{equation}{section}
\numberwithin{figure}{section}
\let\old@ref@lyx\ref
\newcommand{\ref@starred@lyx}[1]{\old@ref@lyx*{#1}}
\newcommand{\ref@safe@lyx}{\@ifstar\ref@starred@lyx\ref@starred@lyx}

\@ifpackageloaded{hyperref}{
\@ifpackageloaded{prettyref}{
\global\let\prettyref@old@lyx\prettyref
\global\def\prettyref#1{{%
\makeatletter%
\def\ref##1{\ref@safe@lyx{##1}}%
\hyperref[#1]{\prettyref@old@lyx{#1}}%
}}
}{}

\global\def\eqref#1{\hyperref[#1]{\textup{(\ref*{#1})}}}
}{}
\usepackage{enumitem}		
\theoremstyle{plain}
\newtheorem{thm}{Theorem}[section]
  \theoremstyle{definition}
  \newtheorem{condition}[thm]{Condition}
  \theoremstyle{remark}
  \newtheorem{rem}[thm]{Remark}
  \theoremstyle{plain}
  \newtheorem{cor}[thm]{Corollary}
  \theoremstyle{plain}
  \newtheorem{lem}[thm]{Lemma}
  \theoremstyle{remark}
  \newtheorem{notation}[thm]{Notation}
  \theoremstyle{definition}
  \newtheorem{defn}[thm]{Definition}

\usepackage[normalem]{ulem}
\newcommand*{\ProofCaption}[1]{\uline{#1}}

\newcommand{\BrSum}{\Bigg}
\newcommand{\BrInt}{\bigg}
\newcommand{\BrSuml}[1]{\mathopen{\BrSum#1}}
\newcommand{\BrSumr}[1]{\mathclose{\BrSum#1}}
\newcommand{\BrIntl}[1]{\mathopen{\BrInt#1}}
\newcommand{\BrIntr}[1]{\mathclose{\BrInt#1}}

\usepackage[utopia,greekuppercase=italicized]{mathdesign} 
\usepackage{esint} 
\newcommand{\uppi}{\piup}
\newcommand{\upDelta}{\Deltaup}
\newcommand{\mathbold}[1]{\boldsymbol{#1}}

\@ifpackageloaded{prettyref}{
\newrefformat{def}{Definition~\ref{#1}}
\newrefformat{lem}{Lemma~\ref{#1}}
\newrefformat{thm}{Theorem~\ref{#1}}
\newrefformat{cor}{Corollary~\ref{#1}}
\newrefformat{pro}{Proposition~\ref{#1}}
\newrefformat{prop}{Proposition~\ref{#1}}
\newrefformat{not}{Notation~\ref{#1}}
\newrefformat{exa}{Example~\ref{#1}}
\newrefformat{cha}{Chapter~\ref{#1}}
\newrefformat{chap}{Chapter~\ref{#1}}
\newrefformat{rem}{Remark~\ref{#1}}
\newrefformat{sec}{\S\ref{#1}}
\newrefformat{sub}{\S\ref{#1}}
\newrefformat{tab}{Table~\ref{#1}}
\newrefformat{fig}{Figure~\ref{#1}}
\newrefformat{con}{Condition~\ref{#1}}
}{ }

\usepackage[final]{microtype}

\renewcommand{\epsilon}{\varepsilon}
\renewcommand{\phi}{\varphi}
\renewcommand{\leq}{\leqslant}
\renewcommand{\geq}{\geqslant}

\@ifpackageloaded{natbib}{\setcitestyle{semicolon, aysep={,}}}{ }

\providecommand{\SortNoop}[1]{}

\makeatother

\begin{document}

\global\long\def\coloneqq{\mathrel{\mathop:}=}

\global\long\def\eqqcolon{=\mathrel{\mathop:}}

\global\long\def\colonsp{\hphantom{:}}

\global\long\def\comma{,\,}

\global\long\def\impl{\implies\;}

\global\long\def\vect#1{\mathbold{#1}}

\global\long\def\vectold#1{\mathbf{#1}}

\global\long\def\vectg#1{\boldsymbol{#1}}

\global\long\def\qedfix{\tag*{{\qedhere}}}

\global\long\def\e{\mathalpha{\mathrm{e}}}

\global\long\def\i{\mathalpha{\mathrm{i}}}

\global\long\def\zerovect{\mathbf{0}}

\global\long\def\laplacian{\upDelta}

\global\long\def\upi{\uppi}

\global\long\def\half{\tfrac{1}{2}}

\global\long\def\halfx{\frac{1}{2}}

\global\long\def\brsum#1{\BrSuml(#1\BrSumr)}

\global\long\def\brsuml{\BrSuml(}

\global\long\def\brsumr{\BrSumr)}

\global\long\def\brint#1{\BrIntl(#1\BrIntr)}

\global\long\def\brintl{\BrIntl(}

\global\long\def\brintr{\BrIntr)}

\global\long\def\sqsum#1{\BrSuml[#1\BrSumr]}

\global\long\def\sqint#1{\BrIntl[#1\BrIntr]}

\global\long\def\abs#1{\lvert#1\rvert}

\global\long\def\absx#1{\left\lvert #1\right\rvert }

\global\long\def\absb#1{\bigl|#1\bigr|}

\global\long\def\abssum#1{\biggl|#1\biggr|}

\global\long\def\absint#1{\BrIntl|#1\BrIntr|}

\global\long\def\norm#1{\lVert#1\rVert}

\global\long\def\normx#1{\left\lVert #1\right\rVert }

\global\long\def\bnorm#1{\bigl\lVert#1\bigr\rVert}

\global\long\def\normsum#1{\biggr\Vert#1\biggr\Vert}

\global\long\def\thickvbar{\,|\!\!|\!\!|\!\!|\!\!|\!\!|\!\!|\,}

\global\long\def\normbar#1{\mathopen{\thickvbar}#1\mathclose{\thickvbar}}

\global\long\def\inprod#1{\langle#1\rangle}

\global\long\def\japbrack#1{\langle#1\rangle}

\global\long\def\floor#1{\lfloor#1\rfloor}

\global\long\def\floorx#1{\left\lfloor #1\right\rfloor }

\global\long\def\ceil#1{\lceil#1\rceil}

\global\long\def\atvalue#1{\rvert_{#1}}

\global\long\def\Re{\operatorname{Re}}

\global\long\def\Im{\operatorname{Im}}

\global\long\def\tr{\operatorname{tr}}

\global\long\def\dist{\operatorname{dist}}

\global\long\def\esssupp{\operatorname{ess\, supp}}

\global\long\def\supp{\operatorname{supp}}

\global\long\def\sgn{\operatorname{sgn}}

\global\long\def\diam{\operatorname{diam}}

\global\long\def\dif{\operatorname{d}\!}

\global\long\def\xd#1{\mathrm{d}#1\,}

\global\long\def\trlim{\operatorname*{\mathfrak{G}_{1}-lim}}

\global\long\def\op{\operatorname{op}}

\global\long\def\normnew#1#2{\mathsf{M}^{#1}(#2)}

\global\long\def\transp#1{#1^{\mathrm{T}}}

\global\long\def\closure#1{\overline{#1}}

\global\long\def\comp#1{#1^{\mathrm{c}}}

\global\long\def\multind#1{\vect{#1}}

\global\long\def\multindg#1{\vect{#1}}

\global\long\def\diff#1{\mathop{\mathrm{d}#1}}

\global\long\def\meas#1{\mu_{#1}}

\global\long\def\mudiff#1#2{\mathop{\meas{#2}(\mathrm{d}#1)}}

\global\long\def\dif{DIF}

\global\long\def\window{\phi}

\global\long\def\opawshort{T_{r}}

\global\long\def\Wig{\mathcal{W}}

\global\long\def\Wigner{\Wig_{\window}}

\global\long\def\Wigtwo{\Wig_{\window_{2},\window_{1}}}

\global\long\def\W{W}

\global\long\def\Q{Q}

\global\long\def\Qn{\Q_{\vect n(\vect u)}}

\global\long\def\Qo{\Q_{\vect{\omega}}}

\global\long\def\opwone{\operatorname{op}}

\global\long\def\diag{\operatorname{diag}}

\global\long\def\tub{\operatorname{tub}}

\global\long\def\symdiff{\upDelta}

\global\long\def\stftadj{\makebox[0pt][l]{\ensuremath{\mathcal{F}_{\window}}}\phantom{\mathcal{F}}^{\ast}}

\global\long\def\stftadjt{\makebox[0pt][l]{\ensuremath{\phantom{\mathcal{F}}^{\ast}}}\mathcal{F}_{\window_{2}}}

\title{\vspace{-4ex}
{\LARGE{}Two-term Szeg\H{o} theorem for generalised anti-Wick operators}}

\author{J.~P.~Oldfield}
\maketitle
\begin{abstract}
This article concerns the asymptotics of pseudodifferential operators
whose Weyl symbol is the convolution of a discontinuous function dilated
by a large scaling parameter with a smooth function of constant scale.
These operators include as a special case generalised anti-Wick operators,
also known as Gabor\textendash{}Toeplitz operators, with smooth windows
and dilated discontinuous symbol. The main result is a two-term Szeg\H{o}
theorem, that is, the asymptotics of the trace of a function of the
operator. A special case of this is the asymptotic terms of the eigenvalue
counting function. In both cases, previously only the first term in
the asymptotic expansion was known explicitly.
\end{abstract}

\section{Introduction}

We consider pseudodifferential operators whose symbol $p$ is dilated
by a large scaling parameter $r$ and ``smoothed out'' by a convolution
factor $\W\in\mathcal{S}(\mathbb{R}^{2d})$ whose integral is $1$.
Explicitly, using the Weyl quantisation (see \prettyref{sub:Preliminaries}),
we consider operators acting on $L^{2}(\mathbb{R}^{d})$ of the form
\[
\opawshort[p]\coloneqq\opwone[\W\ast p_{r}],\qquad\text{where for }\vect z\in\mathbb{R}^{2d}\text{ we set }p_{r}(\vect z)\coloneqq p(\vect z/r).
\]
The main interest in these operators arises from generalised anti-Wick
operators. The \emph{generalised anti-Wick operator} with windows
$\window_{1},\window_{2}\in L^{2}(\mathbb{R}^{d})$ and symbol $p\in L^{\infty}(\mathbb{R}^{d})$
is defined to be $\stftadjt p\mathcal{F}_{\window_{1}}$, where $\mathcal{F}_{\window}\colon L^{2}(\mathbb{R}^{d})\to L^{2}(\mathbb{R}^{2d})$
is the short-time Fourier transform and $p$ acts by multiplication
on $L^{2}(\mathbb{R}^{2d})$ (see \prettyref{sub:GT-operators}).
These are a special case of operators the form $\opawshort[p]$; that
is, for suitable windows $\window_{1},\window_{2}$, there is a corresponding
$\W$ such that
\[
\opawshort[p]=\stftadjt p_{r}\mathcal{F}_{\window_{1}},\qquad\text{where for }\vect z\in\mathbb{R}^{2d}\text{ we set }p_{r}(\vect z)\coloneqq p(\vect z/r).
\]

The result is part of the asymptotic expansion of $\tr f(\opawshort[a\chi_{\Omega}])$
as $r\to\infty$, where $\Omega\subseteq\mathbb{R}^{2d}$, $a$ is
a function acting on $\mathbb{R}^{2d}$ and $f$ is a function such
that $f(0)=0$. Since the symbol $a\chi_{\Omega}$ is discontinuous,
this is referred to as a Szeg\H{o}-type expansion in analogue with
such formulae for Toeplitz matrices. Such a result is already known
(discussed below) but only for the first term, which is a standard
Weyl-type term of order $r^{2d}$. The result proved here (see \prettyref{sub:szego})
gives an explicit expression for the second asymptotic term, which
is a boundary-related term of order $r^{2d-1}$.

An important special case of this is where $a\equiv1$ and $f$ is
an indicator function, which gives the asymptotics of the eigenvalue
counting function of $\opawshort[\chi_{\Omega}]$. The first term
of this expansion shows how many eigenvalues are close to $1$, and
the second term shows how many eigenvalues are between $0$ and $1$
(in what is sometimes called the ``plunge region''). This gives
some quantitative detail to the idea that $\opawshort[\chi_{\Omega}]$
acts somewhat like a projection, in that it ``projects'' the time-frequency
representation of functions on to $\Omega$. As with the general result,
previously known results about the eigenvalue counting function (discussed
below) only give an explicit expression for the first asymptotic term,
whereas the result proved here (see \prettyref{sub:eigenvalue}) gives
an explicit expression for the second term.

The semiclassical calculus for operators whose Weyl symbol is smooth
is already well known \citep[Theorem~(III-11)]{robert1987}. However,
although the Weyl symbol of interest here $\W\ast p_{r}$ is smooth,
even when $p$ is discontinuous, it is not of the correct asymptotic
form to apply that theory. In the terminology of \citet[Definition~(II-13)]{robert1987},
it is not an $h$\nobreakdash-admissible operator (with the natural
choice of $h=1/r^{2}$). The problem is that symbol depends upon two
different scales in the phase space variable $\vect z$: when $\vect z$
is far from the boundary of $\Omega$, $(a\chi_{\Omega})_{r}(\vect z)$
varies asymptotically like $a_{r}(\vect z)$, so changes in $\vect z$
proportional to $r$ are important; when $\vect z$ is near to the
boundary it varies like $W\ast\chi_{r\Omega}(\vect z)$, so changes
in $\vect z$ on a constant scale are important.

The proof (outlined more precisely in \prettyref{sub:proof-overview})
begins in a similar way to that of the usual semiclassical calculus:
we prove a Weyl composition result with the usual formula for the
approximating symbol, but the remainder is shown to satisfy trace
norm and operator norm bounds that are more delicate then usual (\prettyref{lem:composition-norm-bounds}).
The author hopes that these estimates may be of independent interest.
In \prettyref{sub:composition} this result is proved and combined
with facts about the geometry of $\Omega$ to show that we may compose
$\opawshort[a\chi_{\Omega}]$ with itself with sufficiently small
remainder. In \prettyref{sub:trace-asymptotics} the trace asymptotics
of the resulting operator are established using further geometrical
facts. The relevant geometrical theory of tubular neighbourhoods is
collected in \prettyref{sec:appendix-tubular}.

\paragraph{Related Szeg\H{o}-type theorems.}

The original Szeg\H{o} theorems are results about the asymptotic expansion
of $\log\det T_{n}$ (that is, $\tr\log T_{n}$) as $n\to\infty$,
where $T_{n}$ is an $n\times n$ Toeplitz matrix (\citealp{szego1915grenzwertsatz,szego1952};
see also \citealp[Chapter~5]{szegogrenander1958}). (The parameter
$r$ used here is analogous to $n^{2}$ in such problems.) Similar
theorems have been proved for Wiener-Hopf operators, which are a continuous
analogue of Toeplitz operators: whereas Toeplitz operators involve
discrete convolution with a sequence and truncation to a finite length,
Wiener-Hopf operators involve the standard convolution with a function
and truncation to a bounded domain.Szeg\H{o} theorems for both types
of operator have been the subject of extensive study; see, for example,
\citet{bottcher2006analysis}. The intention here is just to highlight
a few of the most directly relevant results.

A generalization of Wiener-Hopf operators is pseudodifferential operators
with discontinuous symbol (without the convolution factor $\W$ as
in the operators considered here, and usually with the left quantisation
rather than the Weyl quantisation). If $\widetilde{T}_{r}$ is a pseudodifferential
operator with symbol of the form $a(\vect x,\vect{\xi})\chi_{\Omega}(\vect x)$
where $a$ is smooth, i.e.\ the discontinuity is in the configuration
variable but not the frequency variable, the complete asymptotic expansion
of $\tr f(\widetilde{T}_{r})$ is known for quite general functions
\citep{widom1985}. The terms in this expansion are of the order $r^{2(d-k)}$,
where $k$ takes non-negative integer values. The coefficients depend
on the geometry of $\Omega$, and it is possible to obtain geometrical
insights into these coefficients \citep{roccaforte1984,roccaforte2013b}
by using geometrical ideas broadly similar to the ones used in this
paper, particularly \prettyref{lem:tube-change-variables}.

When there is also a discontinuity in the frequency variable, i.e.\ the
symbol is of the form $a\chi_{\Omega_{1}\times\Omega_{2}}$ where
$\Omega_{1},\Omega\subseteq\mathbb{R}^{d}$, two terms of the asymptotic
expansion of $\tr f(\widetilde{T}_{r})$ are known (\citealp{widom1982};
\citealp{sobolev2013quasifinal,sobolev2013corners}). The first term
is equal to the one in the result proved here (in particular it is
of order $r^{2d}$). However, the second term is of order $r^{2d-2}\log r$
and depends on the value of $a$ on $\partial\Omega_{1}\times\partial\Omega_{2}$,
in contrast to the result proved here where the second term is of
order $r^{2d-1}$ and depends on the value of $a$ on $\partial\Omega$.

For generalised anti-Wick operators, which are a subclass of the operators
$\opawshort[p]$ considered here, a one-term Szeg\H{o} theorem was
found by \citet{feichtinger2001}. (They called these operators \emph{Gabor\textendash{}Toeplitz
operators}.) Compared to the requirements here, their regularity requirements
are very mild: the symbol merely has to be in $L^{1}\cap L^{\infty}$,
rather than possessing a discontinuity of the specific form $\chi_{\Omega}$,
and the window function merely has to be in $L^{2}(\mathbb{R}^{d})$
rather than $\mathcal{S}(\mathbb{R}^{d})$. However the symbol must
also be positive and the two windows must be equal, which implies
that the operator is positive. That result is for the first term in
the asymptotic expansion, with $o(r^{2d})$ remainder.

\paragraph{Related eigenvalue counting function results.}

The asymptotics of the eigenvalue counting function is a consequence
of the Szeg\H{o} theorem for $\opawshort[\chi_{\Omega}]$, but has
also been studied in its own right.

Anti-Wick operators (which are generalised anti-Wick operators with
Gaussian windows) were first studied systematically by \citet{berezin1971wick}.
This included a result (Theorem~12 of that paper) giving one asymptotic
term of the eigenvalue counting function in roughly the inverse situation
to the one of interest here: he considered eigenvalues below a fixed
value, for symbols that are bounded below by a positive value.

Anti-Wick operators were introduced into the time\textendash{}frequency
community by a paper of \citet{daubechies1988time}, which she called
\emph{time\textendash{}frequency localization operators} when the
symbol is an indicator function. This included two asymptotic terms
of the eigenvalue counting function (Remark~2 and Remark~3 in §IV.B
of that paper) for a specific operator: the anti-Wick operator whose
symbol is the indicator function of the unit disc. She proved this
by explicitly finding the eigenvalues and eigenfunctions of this operator,
using the fact that these are known for Weyl pseudodifferential operators
with spherically symmetric symbols.

For generalised anti-Wick operators whose symbol is a general indicator
function, only the first asymptotic term of the eigenvalue counting
function was previously known. This was shown for one dimensional
operators by \citet[Theorem~2 and Corollary~1]{ramanathan1994}, and
in higher dimensions by \citet[Corollary~2.3 and Comment~(iii) in §2]{feichtinger2001}
using their Szeg\H{o} result. \Citet*[Example~(a) on p.~731]{demari2002}
showed that the asymptotic order of the second term is $r^{2d-1}$
(including a lower bound for it), but did not find an explicit expression.

\paragraph{Notation}

Here are a few notational conventions used throughout. We denote the
space of Schwartz functions on $\mathbb{R}^{m}$ by $\mathcal{S}(\mathbb{R}^{m})$.
The function $\chi_{\Lambda}$ is the indicator function of a set
$\Lambda$. We denote the $k$\nobreakdash-dimensional Hausdorff
measure by $\meas k$; in particular $\mudiff{\vect u}{m-1}$ is the
surface element in $\mathbb{R}^{m}$, and when $k$ equals the ambient
dimension $\meas k$ is simply the Lebesgue measure. The set of natural
numbers including zero is denoted by $\mathbb{N}_{0}$, so that the
set of $m$\nobreakdash-dimensional multi-indices is $\mathbb{N}_{0}^{m}$.
The boundary of a set $\Omega$ is denoted by $\partial\Omega$ and
its complement by $\comp{\Omega}$. The tubular radius $\tau(\partial\Omega)$
and tubular neighbourhood $\tub(\partial\Omega,t)$ of $\partial\Omega$
are defined in \prettyref{sub:tubular-basic}.

\paragraph{Acknowledgements.}

It is the author's pleasure to thank A.\hspace{1pt}V. Sobolev for
suggesting the problem and his tireless support, especially his observation
that a rougher version of \prettyref{lem:composition-norm-bounds}
(similar to \citealp[Lemma~3.12 and Corollary~3.13]{sobolev2013quasifinal})
could give the asymptotics when $\partial\Omega$ is straight.

The author would also like to thank the organisers of the workshop
on phase space methods for pseudodifferential operators at the Erwin
Schrödinger Institute in October 2012, where he had many productive
discussions. This included a conversation with K.~Nowak, who the
author would like to thank for informing him of a two-term Szeg\H{o}
theorem that K.~Nowak and H.~G.~Feichtinger have made progress
on under somewhat different conditions to those considered here.

This work was supported by the Engineering and Physical Sciences Research
Council {[}grant number EP/P505771/1{]}.

\section{Statement of results\label{sec:results}}

\subsection{\label{sub:Preliminaries}Weyl quantisation preliminaries}

We will use the \emph{Weyl quantisation}: for a suitable symbol $q$,
we define the operator $\opwone[q]$ for each $u\in\mathcal{S}(\mathbb{R}^{2d})$
by 
\[
(\opwone[q]u)(\vect x)\coloneqq\frac{1}{(2\upi)^{d}}\int_{\mathbb{R}^{d}}\int_{\mathbb{R}^{d}}\e^{\i(\vect x-\vect y)\cdot\vect{\xi}}q(\half(\vect x+\vect y),\vect{\xi})u(\vect y)\diff{\vect y}\diff{\vect{\xi},}
\]
and extend this to $L^{2}(\mathbb{R}^{d})$ by density. This satisfies
the operator norm and trace norm estimates (\citealp[Corollary~2.5(i)]{boulkhemair1999}
and \citealp[Theorem~9.4]{dimassi1999spectral} respectively)
\[
\norm{\opwone[q]}\leq C_{d}\sum_{\abs{\multind k}\leq d+2}\norm{\partial^{\multind k}q}_{L^{\infty}(\mathbb{R}^{2d})},\qquad\norm{\opwone[q]}_{1}\leq C_{d}^{\prime}\sum_{\abs{\multind k}\leq2d+1}\norm{\partial^{\multind k}q}_{L^{1}(\mathbb{R}^{2d})},
\]
where $C_{d}$ and $C_{d}'$ are constants depending only on the dimension.
(This operator norm estimate is slightly weaker than the one in the
cited work, but is sufficient for our purposes.) When the trace norm
estimate is finite, the trace exists and equals
\[
\tr\opwone[q]=\frac{1}{(2\upi)^{d}}\int_{\mathbb{R}^{2d}}q(\vect z)\diff{\vect z.}
\]
The adjoint of the operator is given by \citep[Proposition~(2.6)]{folland1989harmonic}
\[
(\opwone[q])^{*}=\opwone[\overline{q}];
\]
in particular, if $q$ is real-valued then $\op[q]$ is self-adjoint.

As stated in the introduction, the operators of interest here depend
on a discontinuous symbol $p$, dilated by a factor $r$ and convolved
with a Schwartz function $\W\in\mathcal{S}(\mathbb{R}^{2d})$, so
that
\[
\opawshort[p]\coloneqq\opwone[\W\ast p_{r}],\qquad\text{where for }\vect z\in\mathbb{R}^{2d}\text{ we set }p_{r}(\vect z)\coloneqq p(\vect z/r).
\]
Applying the Weyl operator norm and trace norm estimates to $\opawshort[p]$
we obtain
\[
\norm{\opawshort[p]}\leq C_{d}\sum_{\abs{\multind k}\leq d+2}\norm{\partial^{\multind k}W}_{L^{1}(\mathbb{R}^{2d})}\norm p_{L^{\infty}(\mathbb{R}^{2d})},\qquad\norm{\opawshort[p]}_{1}\leq C_{d}^{\prime}r^{2d}\sum_{\abs{\multind k}\leq2d+1}\norm{\partial^{\multind k}W}_{L^{1}(\mathbb{R}^{2d})}\norm p_{L^{1}(\mathbb{R}^{2d})}.
\]
Since we will be interested in the effects of varying the scale of
the discontinuous part of the symbol, rather than varying $\W$, we
will often use the notation 
\[
x\lesssim y\qquad\Longleftrightarrow\qquad\text{there exists }C_{\W}>0\text{ such that }x\leq C_{\W}y,
\]
where $C_{\W}$ is some constant depending only on $\W$ and the dimension
$d$ (not on $p$ or $r$). Using this notation, the above inequalities
are 
\[
\norm{\opawshort[p]}\lesssim\norm p_{L^{\infty}(\mathbb{R}^{2d})},\qquad\norm{\opawshort[p]}_{1}\lesssim r^{2d}\norm p_{L^{1}(\mathbb{R}^{2d})}.
\]
The trace formula, combined with the fact that the integral of $\W$
is $1$, gives
\[
\tr\opawshort[p]=\frac{r^{2d}}{(2\upi)^{d}}\int_{\mathbb{R}^{2d}}p(\vect z)\diff{\vect z.}
\]

\subsection{Szeg\H{o} theorem\label{sub:szego}}

In this subsection we state \prettyref{thm:main}, the Szeg\H{o} theorem
for operators of the form $\opawshort[a\chi_{\Omega}]$. It has the
following regularity conditions on the symbol.
\begin{condition}
\label{con:symbol-and-domain}Let all of the following be satisfied.
\begin{itemize}
\item Let $\W\in\mathcal{S}(\mathbb{R}^{2d})$ satisfy $\int_{\mathbb{R}^{2d}}\W(\vect z)\diff{\vect z}=1$.
\item Let $\Omega\subseteq\mathbb{R}^{2d}$ have $C^{2}$ boundary such
that $\partial\Omega$ has a tubular neighbourhood (see \prettyref{sec:appendix-tubular}).
\item Let $a$ be a twice continuously differentiable function on $\mathbb{R}^{2d}$
satisfying $\partial^{\multind k}a\in L^{1}(\mathbb{R}^{2d})\cap L^{\infty}(\mathbb{R}^{2d})$
for all $\multind k\in\mathbb{N}_{0}^{2d}$ such that $\abs{\multind k}\leq2$.
\end{itemize}
\end{condition}
\begin{rem}
Whenever \prettyref{con:symbol-and-domain} is satisfied we can conclude
that that $a$ satisfies the boundary integrability properties $\partial^{\multind k}a\in L^{1}(\partial\Omega)$
for $\abs{\multind k}\leq1$. This can be seen by applying \prettyref{lem:boundary-value-approx}
with $g\equiv1$.
\end{rem}

We also need a condition on the regularity of $f$. This depends on
whether we define $f(\opawshort[a\chi_{\Omega}])$ using the holomorphic
functional calculus or the Borel functional calculus. In the latter
case we impose additional restrictions on $\W$ and $a$ to ensure
that the operator $\opawshort[a\chi_{\Omega}]$ is self-adjoint (by
ensuring that its Weyl symbol is real).
\begin{condition}
\label{con:spectral-function}For functions $a$ and $\W$, let $f$
be a function satisfying $f(0)=0$ and one of the following.
\begin{enumerate}
\item Let $f$ be a holomorphic function on $\mathbb{C}$.
\item Let $a$ be real-valued, let $\W$ be real-valued and let $f$ be
an infinitely differentiable function on $\mathbb{R}$.
\end{enumerate}
\end{condition}
The boundary term depends on a type of directional antiderivative
of $\W$. Specifically, for any $\W\in\mathcal{S}(\mathbb{R}^{2d})$
with $\int_{\mathbb{R}^{2d}}\W(\vect z)\diff{\vect z}=1$, we define
\[
\Qo(\lambda)\coloneqq\int_{\{\vect z\in\mathbb{R}^{2d}:\vect z\cdot\vect{\omega}\leq\lambda\}}\W(\vect z)\diff{\vect z}\qquad(\vect{\omega}\in\mathbb{S}^{2d-1}).
\]
This satisfies $\lim_{\lambda\to\infty}\Qo(\lambda)=\int_{\mathbb{R}^{2d}}\W(\vect z)\diff{\vect z}=1$,
and so
\[
1-\Qo(\lambda)=\int_{\{\vect z\in\mathbb{R}^{2d}:\vect z\cdot\vect{\omega}\geq\lambda\}}\W(\vect z)\diff{\vect z}\qquad(\vect{\omega}\in\mathbb{S}^{2d-1}).
\]

\begin{thm}
\label{thm:main}Let $\W$, $a$, $\Omega$, $f$ satisfy \prettyref{con:symbol-and-domain}
and \prettyref{con:spectral-function}. Then 
\[
\tr f(\opawshort[a\chi_{\Omega}])=r^{2d}A_{0}(a,\Omega,f)+r^{2d-1}A_{1}(a,\Omega,f;\W)+O(r^{2d-2})
\]
as $r\to\infty$, where
\begin{align*}
A_{0}(a,\Omega,f) & =\frac{1}{(2\upi)^{d}}\int_{\Omega}f(a(\vect z))\diff{\vect z},\\
A_{1}(a,\Omega,f;\W) & =\frac{1}{(2\upi)^{d}}\int_{\partial\Omega}\int_{\mathbb{R}}\Bigl(f(\Qn(\lambda)a(\vect u))-\Qn(\lambda)f(a(\vect u))\Bigr)\diff{\lambda}\mudiff{\vect u}{2d-1}.
\end{align*}

\end{thm}
The proof is given in \prettyref{sec:Proof}, including an overview
in \prettyref{sub:proof-overview}. In the case of generalised anti-Wick
operators, the conditions and conclusions can be explicitly expressed
in terms of the windows instead of $\W$; see \prettyref{sub:GT-operators}.

We now observe why the quantities in \prettyref{thm:main} are well
defined. The operator $f(\opawshort[a\chi_{\Omega}])$ is trace class
since, using the fact that $f(0)=0$, we have
\[
\norm{f(\opawshort[a\chi_{\Omega}])}_{1}\leq\norm{\opawshort[a\chi_{\Omega}]}_{1}\sup_{\abs t\leq\norm{\opawshort[a\chi_{\Omega}]}}\abs{f'(t)},
\]
and the bounds in \prettyref{sub:Preliminaries} show that this is
finite. Let 
\[
\Q_{\mathrm{max}}\coloneqq\sup_{\vect{\omega}\in\mathbb{S}^{2d-1}}\sup_{\lambda\in\mathbb{R}\vphantom{\mathbb{S}^{2d-1}}}\abs{\Qo(\lambda)},
\]
which in particular satisfies $\Q_{\max}\leq\int_{\mathbb{R}^{2d}}\abs{W(\vect z)}\diff{\vect z}$.
The two asymptotic terms are absolutely integrable with bounds
\begin{align*}
\abs{A_{0}(a,\Omega,f)} & \leq\frac{1}{(2\upi)^{d}}\norm a_{L^{1}(\Omega)}\sup_{\abs t\leq\norm a_{L^{\infty}(\Omega)}}\abs{f'(t)},\\
\abs{A_{1}(a,\Omega,f;\W)} & \leq\frac{2}{(2\upi)^{d}}\norm a_{L^{1}(\partial\Omega)}\int_{\mathbb{R}^{2d}}\abs{\vect z'\W(\vect z')}\diff{\vect z'}\sup_{\abs t\leq\Q_{\mathrm{max}}\norm a_{L^{\infty}(\partial\Omega)}}\abs{f'(t)}.
\end{align*}
The bound on $A_{0}$ is immediate, and the bound on $A_{1}$ uses
the easily checked fact that for any $\vect{\omega}\in\mathbb{S}^{2d-1}$
we have
\[
\int_{\mathbb{R}}\abs{\Qo(\lambda)-\chi_{[0,\infty)}(\lambda)}\diff{\lambda}\leq\int_{\mathbb{R}^{2d}}\abs{\vect{\omega}\cdot\vect z'\W(\vect z')}\diff{\vect z'.}
\]

\subsection{Eigenvalue counting function\label{sub:eigenvalue}}

In this subsection we give a precise statement of the special case
discussed in the introduction: two terms of the asymptotic expansion
of the eigenvalue counting function for operators of the form $\opawshort[\chi_{\Omega}]$.
We use the notation $N(\opawshort[\chi_{\Omega}],[\delta,\infty))$
to mean the number of eigenvalues of $\opawshort[\chi_{\Omega}]$
in the interval $[\delta,\infty).$ The proof is a standard approximation
argument applied to \prettyref{thm:main}, and is detailed at the
end of this subsection.
\begin{cor}
\label{cor:main}Let $\Omega\subseteq\mathbb{R}^{2d}$ be a compact
set with $C^{2}$ boundary and $\delta\in(0,1)$. Let $\W\in\mathcal{S}(\mathbb{R}^{2d})$
be real valued and satisfy $\int_{\mathbb{R}^{2d}}\W(\vect z)\diff{\vect z}=1$
and 
\[
\forall\vect{\omega}\in\mathbb{S}^{2d-1}\text{ have }\mu_{1}(\{\lambda\in\mathbb{R}:\Qo(\lambda)=\delta\})=0.
\]
Then
\[
N(\opawshort[\chi_{\Omega}],[\delta,\infty))=r^{2d}A_{0}(1,\Omega,\chi_{[\delta,\infty)})+r^{2d-1}A_{1}(1,\Omega,\chi_{[\delta,\infty)};\W)+o(r^{2d-1})
\]
as $r\to\infty$. Specifically, these terms satisfy
\begin{align*}
A_{0}(1,\Omega,\chi_{[\delta,\infty)}) & =\frac{1}{(2\upi)^{d}}\meas{2d}(\Omega),\\
A_{1}(1,\Omega,\chi_{[\delta,\infty)};\W) & =\frac{1}{(2\upi)^{d}}\int_{\partial\Omega}g_{\vect n(\vect u)}(\lambda\delta)\mudiff{\vect u}{2d-1},
\end{align*}
where for each $\delta\in(0,1)$, $\vect{\omega}\in\mathbb{S}^{2d-1}$
we set
\[
g_{\vect{\omega}}(\delta)\coloneqq\meas 1(\{\lambda\in(-\infty,0]:\Qo(\lambda)>\delta\})-\meas 1(\{\lambda\in[0,\infty):\Qo(\lambda)<\delta\}).
\]
\end{cor}
\begin{rem}
\label{rem:increasing-Q}The statement of \prettyref{cor:main} is
somewhat simpler when that $\Qo$ is a non-decreasing function for
all $\vect{\omega}\in\mathbb{S}^{2d-1}$. A sufficient condition for
this is that $\W$ is non-negative (for another sufficient condition
see \prettyref{rem:GT-Q-expression}). In this case:
\begin{itemize}
\item The condition relating $\Qo$ and $\delta$ holds if and only if for
each $\vect{\omega}\in\mathbb{S}^{2d-1}$ there exists a unique $\lambda\in\mathbb{R}$
such that $\Qo(\lambda)=\delta$; we denote such a $\lambda$ by $\Qo^{-1}(\delta)$,
even if $\Qo$ is not invertible on its whole domain.
\item We then have $g_{\vect{\omega}}(\delta)=-\Qo^{-1}(\delta)$, so the
boundary term simplifies to 
\[
A_{1}(1,\Omega,\chi_{[\delta,\infty)};\W)=-\frac{1}{(2\upi)^{d}}\int_{\partial\Omega}\Qn^{-1}(\delta)\mudiff{\vect u}{2d-1}.
\]

\end{itemize}
\end{rem}
\begin{proof}[Proof of \prettyref{cor:main}]
We have $N(\opawshort[\chi_{\Omega}],[\delta,\infty))=\tr\chi_{[\delta,\infty)}(\opawshort[\chi_{\Omega}])$;
however, we cannot immediately apply \prettyref{thm:main} with $f\coloneqq\chi_{[\delta,\infty)}$
because this function is not sufficiently smooth to satisfy \prettyref{con:spectral-function}.

Let $\epsilon>0$ such that $\epsilon<\delta$. Let $f_{-\epsilon}$
and $f_{+\epsilon}$ be smooth increasing functions satisfying $f_{\pm\epsilon}(t)=\chi_{[\delta,\infty)}(t)$
except when $t\in(\delta,\delta+\epsilon)$ and $t\in(\delta-\epsilon,\delta)$
respectively. Thus $0\leq f_{-\epsilon}\leq\chi_{[\delta,\infty)}\leq f_{+\epsilon}\leq1$
and 
\[
\tr f_{-\epsilon}(\opawshort[\chi_{\Omega}])\leq\tr\chi_{[\delta,\infty)}(\opawshort[\chi_{\Omega}])\leq\tr f_{+\epsilon}(\opawshort[\chi_{\Omega}]).
\]
Applying \prettyref{thm:main} to $f_{\pm\epsilon}(\opawshort[\chi_{\Omega}])$
(with $a\equiv1$), we have
\begin{align*}
\lim_{r\to\infty}\frac{\tr\bigl((f_{+\epsilon}-f_{-\epsilon})(\opawshort[\chi_{\Omega}])\bigr)}{r^{2d-1}} & =\frac{1}{(2\upi)^{d}}\int_{\partial\Omega}\int_{\mathbb{R}}\Bigl(f_{+\epsilon}(\Qn(\lambda))-f_{-\epsilon}(\Qn(\lambda))\Bigr)\diff{\lambda}\mudiff{\vect u}{2d-1}\\
 & \leq\frac{1}{(2\upi)^{d}}\meas{2d-1}(\partial\Omega)\sup_{\vect u\in\partial\Omega}\mu_{1}(\{\lambda\in\mathbb{R}:\delta-\epsilon\leq\Qn(\lambda)\leq\delta+\epsilon\}).
\end{align*}
The limit of this bound is $0$ as $\epsilon\to0$, so the result
follows.

It remains to show that $A_{1}$ satisfies the given form. First set
\[
\tilde{A}_{1}\coloneqq\frac{1}{(2\upi)^{d}}\int_{\partial\Omega}\int_{\mathbb{R}}\Bigl(\chi_{[\delta,\infty)}(\Qn(\lambda))-\chi_{[0,\infty)}(\lambda)\Bigr)\diff{\lambda}\mudiff{\vect u}{2d-1}.
\]
A straightforward calculation shows that
\[
A_{1}-\tilde{A}_{1}=\frac{1}{(2\upi)^{d}}\int_{\partial\Omega}\vect n(\vect u)\mudiff{\vect u}{2d-1}\cdot\int_{\mathbb{R}^{2d}}\vect z'\W(\vect z')\diff{\vect z',}
\]
which by the divergence theorem is zero. It is easily seen that $\tilde{A}_{1}$
satisfies the stated form. 
\end{proof}

\subsection{Generalised anti-Wick operators\label{sub:GT-operators}}

Define the \emph{short-time Fourier transform} with window $\window\in L^{2}(\mathbb{R}^{d})$
by
\[
\mathcal{F}_{\window}\colon L^{2}(\mathbb{R}^{d})\to L^{2}(\mathbb{R}^{2d}),\qquad\mathcal{F}_{\window}u(\vect x,\vect{\xi})\coloneqq\frac{1}{(2\upi)^{d/2}}\int_{\mathbb{R}^{d}}\e^{-\i\vect y\cdot\vect{\xi}}u(\vect y)\overline{\window(\vect y-\vect x)}\diff{\vect y.}
\]
When the window is the Gaussian function, $\mathcal{F}_{\window}$
is also known as the \emph{Fourier\textendash{}Bros\textendash{}Iagolnitzer
transform}. (See, for example, \citealp[Chapter~3]{grochenig2001foundations}
or \citealp[\S3.1]{martinez2001} for more information.) The \emph{generalised
anti-Wick operator} with symbol $p$ and windows $\window_{1},\window_{2}$
is defined to be $\stftadjt p\mathcal{F}_{\window_{1}}$. These operators
are known under several names, including Gabor\textendash{}Toeplitz
operators, short-time Fourier transform multipliers and time\textendash{}frequency
localization operators. The case where $\window_{1}=\window_{2}$
is most often of interest.

Generalised anti-Wick operators are bounded on $L^{2}(\mathbb{R}^{d})$
when $p\in L^{\infty}(\mathbb{R}^{2d})$ and $\window_{1},\window_{2}\in L^{2}(\mathbb{R}^{2d})$.
Furthermore, if $p$ is constant then the operator is a multiple of
the identity. Specifically, 
\begin{gather*}
\norm{\stftadjt p\mathcal{F}_{\window_{1}}}\leq\norm{\window_{1}}_{L^{2}(\mathbb{R}^{d})}\norm{\window_{2}}_{L^{2}(\mathbb{R}^{d})}\norm p_{L^{\infty}(\mathbb{R}^{2d})},\\
\stftadjt\mathcal{F}_{\window_{1}}=\inprod{\window_{2},\window_{1}}_{L^{2}(\mathbb{R}^{d})}\operatorname{Id}_{L^{2}(\mathbb{R}^{d})}.
\end{gather*}
These relationships can easily be proved from the Fourier inversion
theorem, or see for example \citet[Corollary~3.2.2 and  Corollary~3.2.3]{grochenig2001foundations}.

\prettyref{thm:main} and \prettyref{cor:main} apply to generalised
anti-Wick operators; that is, there exists a suitable $\W$ (depending
on the windows) such that 
\[
\opawshort[p]=\stftadjt p_{r}\mathcal{F}_{\window_{1}},\qquad\text{where for }\vect z\in\mathbb{R}^{2d}\text{ we set }p_{r}(\vect z)\coloneqq p(\vect z/r).
\]
The following two remarks explain how all references to $\W$ in these
results may be replaced by references directly to the windows. Afterwards
we will describe this $\W$ and explain why the remarks are true.
\begin{rem}
\label{rem:GT-conditions}The conditions on $\W$ can be replaced
by requirements on the window functions:
\begin{itemize}
\item For all the conditions on $\W$ in \prettyref{thm:main} and \prettyref{cor:main}
to hold (including that $\W$ is real-valued), it suffices that $\window_{1}=\window_{2}$
(which we write simply as $\window$), $\window\in\mathcal{S}(\mathbb{R}^{2d})$,
and $\norm{\window}_{L^{2}(\mathbb{R}^{d})}=1$.
\item For the conditions on $\W$ in \prettyref{thm:main} to hold \textbf{except
}that $\W$ be real-valued (so we require \prettyref{con:spectral-function}(1),
the holomorphic $f$ case), it suffices that $\window_{1},\window_{2}\in\mathcal{S}(\mathbb{R}^{2d})$
and $\inprod{\window_{2},\window_{1}}_{L^{2}(\mathbb{R}^{d})}=1$.
\end{itemize}
\end{rem}

\begin{rem}
\label{rem:GT-Q-expression}It is possible to express $\Qo$ directly
in terms of the windows. First consider the one-dimensional case.
We will use the fractional Fourier transform $\mathcal{F}^{t}$, defined
for $t\in\mathbb{R}$ using the functional calculus for unitary operators;
thus $\mathcal{F}^{0}=\mathcal{F}^{4}=\operatorname{Id}_{L^{2}(\mathbb{R})}$
and $\mathcal{F}^{1}$ is the usual Fourier transform. We can instead
index by direction $\vect{\omega}\in\mathbb{S}^{1}$, so that  $\mathcal{F}^{(1,0)}=\operatorname{Id}_{L^{2}(\mathbb{R})}$
and $\mathcal{F}^{(0,1)}=\mathcal{F}$. The expression for $\Qo$
is
\[
\Qo(\lambda)=\int_{-\infty}^{\lambda}\mathcal{F}^{\vect{\omega}}\window_{2}(\eta)\overline{\mathcal{F}^{\vect{\omega}}\window_{1}(\eta)}\diff{\eta.}
\]
In the higher-dimensional case, for each $\vect{\omega}\in\mathbb{S}^{2d-1}$
there exists a unitary operator $T_{\vect{\omega}}$ and $\widetilde{\vect{\omega}}\in\mathbb{S}^{d-1}$
such that 
\[
\Qo(\lambda)=\int_{-\infty}^{\lambda}\int_{\{\vect x\in\mathbb{R}^{d}:\vect x\cdot\widetilde{\vect{\omega}}=\eta\}}T_{\vect{\omega}}\window_{2}(\vect x)\overline{T_{\vect{\omega}}\window_{1}(\vect x)}\mudiff{\vect x}{d-1}\diff{\eta.}
\]
In particular, for any dimension, if $\window_{1}=\window_{2}$ then
$\Qo$ is a non-decreasing function and \prettyref{rem:increasing-Q}
applies.
\end{rem}
The key fact that allows us to apply \prettyref{thm:main} and \prettyref{cor:main}
to generalised anti-Wick operators is their connection to the Weyl
transform, given by 
\[
\stftadjt p\mathcal{F}_{\window_{1}}=\opwone[\Wigtwo\ast p],\qquad\Wigtwo(\vect x,\vect{\xi})=\frac{1}{(2\upi)^{d}}\int_{\mathbb{R}^{d}}\e^{-\i\vect t\cdot\vect{\xi}}\window_{2}(\vect x+\half\vect t)\overline{\window_{1}(\vect x-\half\vect t)}\diff{\vect t.}
\]
The function $\Wigtwo$ is called the \emph{Wigner transform} of $\window_{2},\window_{1}$.
This relationship can be found for example in \citet[Proposition~(3.5)]{folland1989harmonic}
when $\window_{1}=\window_{2}$ or \citet*[Lemma~2.4]{bogg2004}.

\begin{proof}[Proof of \prettyref{rem:GT-conditions}.]
We use the following properties of $\Wigtwo$.
\begin{enumerate}
\item If $\window_{1},\window_{2}\in\mathcal{S}(\mathbb{R}^{2d})$ then
$\Wigtwo\in\mathcal{S}(\mathbb{R}^{2d})$.
\item For all $\vect x\in\mathbb{R}^{d}$ we have $\int_{\mathbb{R}^{d}}\Wigtwo(\vect x,\vect{\xi})\diff{\vect{\xi}}=\window_{2}(\vect x)\overline{\window_{1}(\vect x)}$.
\item We have $\Wigtwo(\vect z)=\overline{\Wig_{\window_{1},\window_{2}}(\vect z)}$.
\end{enumerate}
These properties follow easily from the definition of $\Wigtwo$ (see
for example \citealp[\S1.8]{folland1989harmonic}). \prettyref{rem:GT-conditions}
is an immediate consequence.
\end{proof}

\begin{proof}[Proof of \prettyref{rem:GT-Q-expression}.]
We first work in one dimension. We use another property of $\Wigtwo$.
\begin{enumerate}[label=4.]
\item For each $\vect{\omega}\in\mathbb{S}^{1}$, let $\sigma_{\!\vect{\omega}}\colon\mathbb{R}^{2}\to\mathbb{R}^{2}$
be the rotation that maps $(1,0)\mapsto\vect{\omega}$; then for all
$\vect z\in\mathbb{R}^{2}$ we have $\Wigtwo(\sigma_{\!\vect{\omega}}\vect z)=\Wig_{\mathcal{F}^{\vect{\omega}}\window_{2},\mathcal{F}^{\vect{\omega}}\window_{1}}(\vect z)$.
\end{enumerate}
In others words, the fractional Fourier transform is the metaplectic
operator corresponding to rotation. For example, in \citet{folland1989harmonic}
see Proposition~(1.94)(c) for the $\vect{\omega}=(1,0)$ case (the
usual Fourier transform) and Chapter~4 for discussion of metaplectic
operators (especially Proposition~(4.28) for the relationship to
the Wigner transform). For more information on the fractional Fourier
transform, see for example \citet*{ozaktas2001fractional}.

Combining property~4 with property~2 we obtain
\[
\int_{\{\vect z'\in\mathbb{R}^{2}:\vect z'\cdot\vect{\omega}=\lambda\}}\Wigtwo(\vect z')\mudiff{\vect z'}1=\mathcal{F}^{\vect{\omega}}\window_{2}(\lambda)\overline{\mathcal{F}^{\vect{\omega}}\window_{1}(\lambda)}.
\]
(This is sometimes call the Radon\textendash{}Wigner transform, since
it is the Radon transform of the Wigner distribution.) But $\Qo$
is the antiderivative of this expression, so \prettyref{rem:GT-Q-expression}
is immediate from this. For higher dimensions, we apply similar reasoning
component-wise, so that $T_{\vect{\omega}}$ is the composition of
component-wise fractional Fourier transform operators.
\end{proof}

\section{Proof\label{sec:Proof}}

\subsection{Overview\label{sub:proof-overview}}

There are two steps to the proof of \prettyref{thm:main}, which are
distilled into the two lemmas in this subsection.

To avoid dealing with the scaling parameter $r$ throughout the whole
proof, we will give names to the rescaled versions of $a$ and $\Omega$.
We write
\[
\opawshort[a\chi_{\Omega}]=\op[\W\ast(b\chi_{\Sigma})],\qquad\text{where }b\coloneqq a(\cdot/r),\,\Sigma\coloneqq r\Omega.
\]
The two lemmas will be proved in terms of general $b$, $\Sigma$
without explicit reference to the fact that they are rescaled versions
of other objects. However, in each lemma the remainder scales in such
a way that it is $O(r^{2d-2})$ when $b$ and $\Sigma$ are of this
form.

The first step is composition, where we find an approximation of the
Weyl symbol of $f(\opawshort[a\chi_{\Omega}])$.
\begin{lem}
\label{lem:composition}Let $\W,b,\Sigma,f$ satisfy \prettyref{con:symbol-and-domain}
and \prettyref{con:spectral-function}, and let $\partial\Sigma$
have tubular radius of at least $1$. Then there exists $R$ such
that
\[
\norm{f(\op[\W\ast(b\chi_{\Sigma})])-\op[f(\W\ast(b\chi_{\Sigma}))]}_{1}\leq R(b,\Sigma;\W,f),
\]
where $R$ satisfies the scaling property
\[
R(b,\Sigma;\W,f)=r^{2d-2}R(a,\Omega;\W,f),\qquad\text{for }b=a(\cdot/r),\,\Sigma=r\Omega.
\]

\end{lem}
This is proved in \prettyref{sub:composition}. First, in \prettyref{lem:composition-norm-bounds},
we will prove a trace norm bound for the composition of general Weyl
operators. Next, in \prettyref{lem:composition-square}, we will apply
this to the operator $\op[\W\ast(b\chi_{\Sigma})]$. A naive application
would result in a trace norm bound that includes the four terms 
\[
\norm{\nabla b}_{L^{\infty}(\Sigma)}\norm{\nabla b}_{L^{1}(\Sigma)},\quad\norm{\nabla b}_{L^{\infty}(\Sigma)}\norm b_{L^{1}(\partial\Sigma)},\quad\norm b_{L^{1}(\partial\Sigma)}\norm{\nabla b}_{L^{\infty}(\Sigma)},\quad\norm b_{L^{\infty}(\partial\Sigma)}\norm b_{L^{1}(\partial\Sigma)}.
\]
The first three terms are $O(r^{2d-2})$ as required, but the final
one is $O(r^{2d-1})$. The proof of \prettyref{lem:composition-square}
involves some delicate cancellation using the geometry of $\partial\Sigma$
to obtain a better bound. This completes the proof of \prettyref{lem:composition}
in the case that $f(t)=t^{2}$; at the end of \prettyref{sub:composition}
the proof is given for general $f$.

Combined with the fact that $\abs{\tr A}\leq\norm A_{1}$ for every
trace class operator $A$, \prettyref{lem:composition} tells us that
\[
\tr f(\opawshort[a\chi_{\Omega}])=\tr\op[f(\W\ast(b\chi_{\Sigma}))]+O(r^{2d-2}).
\]
The trace is given by the integral of the Weyl symbol (see \prettyref{sub:Preliminaries}).
The proof of \prettyref{thm:main} is thus completed by finding the
asymptotics of this integral, which is done in the following lemma.
\begin{lem}
\label{lem:trace-asymptotics}Let $\W,b,\Sigma,f$ satisfy \prettyref{con:symbol-and-domain}
and \prettyref{con:spectral-function}. Then there exists $R$ such
that, in the notation of \prettyref{thm:main}, we have
\[
\absint{\int_{\mathbb{R}^{2d}}f(\W\ast(b\chi_{\Sigma})(\vect z))\diff{\vect z}-\,\Bigl(A_{0}(b,\Sigma,f)+A_{1}(b,\Sigma,f;\W)\Bigr)}\leq R(b,\Sigma;\W,f),
\]
where $R$ satisfies the scaling property
\[
R(b,\Sigma;\W,f)=r^{2d-2}R(a,\Omega;\W,f),\qquad\text{for }b=a(\cdot/r),\,\Sigma=r\Omega.
\]

\end{lem}
This is proved in \prettyref{sub:trace-asymptotics}. The proof begins
by noting that, since the integral of $\W$ is $1$ and $f(0)=0$,
we have
\[
\int_{\mathbb{R}^{2d}}f(\W\ast(b\chi_{\Sigma})(\vect z))\diff{\vect z}=\int_{\Sigma}f(b(\vect z))\diff{\vect z}+\int_{\mathbb{R}^{2d}}\Bigl(f(\W\ast(\chi_{\Sigma}b)(\vect z))-\W\ast(\chi_{\Sigma}f(b))(\vect z)\Bigr)\diff{\vect z.}
\]
The first term is simply $A_{0}(b,\Sigma,f)$, which equals $r^{2d}A_{0}(a,\Omega,f)$.
The second term is very similar to $A_{1}(b,\Sigma,f;\W)$; in particular
its integrand is concentrated near to $\partial\Sigma$. However,
unlike $A_{1}$, it is not of the correct asymptotic form; that is,
it does not equal $r^{2d-1}$ multiplied by its unscaled version.
The proof proceeds by using the local geometry of $\partial\Sigma$
to show that this integral is indeed approximately equal to $A_{1}$.

\subsection{Step 1: Composition\label{sub:composition}}

In this subsection we prove \prettyref{lem:composition}, proceeding
as discussed in \prettyref{sub:proof-overview}.
\begin{notation}
In this subsection we frequently decompose vectors $\vect z\in\mathbb{R}^{2d}$
as $\vect z=(\vect z_{1},\vect z_{2})$, where $\vect z_{1},\vect z_{2}\in\mathbb{R}^{d}$.
Furthermore, we use the notation
\[
\japbrack{\vect x}\coloneqq(1+\abs{\vect x}^{2})^{1/2}.
\]

\end{notation}
We start by proving trace norm and operator norm bounds for the error
in replacing the Weyl symbol of composition by a finite number of
terms in the series expansion. In fact we only need the trace norm
bound, and only for $n=0$, but the full result is no harder to prove
and the author hopes that it may be of general interest.

\global\long\def\GG{G}

\begin{lem}
\label{lem:composition-norm-bounds}Let $p,q$ be infinitely differentiable
functions on $\mathbb{R}^{2d}$ such that $\partial^{\multind k}p,\partial^{\multind k}q\in L^{1}(\mathbb{R}^{2d})$
for each $\multind k\in\mathbb{N}_{0}^{2d}$. Let $n\in\mathbb{N}_{0}$,
$\GG\in\mathbb{N}_{0}$. Set
\[
F_{j}(\vect x,\vect y)\coloneqq\frac{\i^{j}}{j!2^{j}}(\nabla_{\vect x_{1}}\cdot\nabla_{\vect y_{2}}-\nabla_{\vect x_{2}}\cdot\nabla_{\vect y_{1}})^{j}(p(\vect x)q(\vect y)),\quad c_{n}(\vect z)\coloneqq\sum_{j=0}^{n}F_{j}(\vect z,\vect z).
\]
Then
\begin{align*}
\norm{\op[p]\op[q]-\op[c_{n}]}_{1} & \leq C_{d,\GG}\sum_{\substack{\multind m\in\mathbb{N}_{0}^{4d}\\
\abs{\multind m}\leq\GG+4d+2
}
}\int_{\mathbb{R}^{2d}}\int_{\mathbb{R}^{2d}}\frac{\abs{\partial^{\multind m}F_{n+1}(\vect x,\vect y)}}{\japbrack{\vect x-\vect y}^{\GG}}\diff{\vect x}\diff{\vect y,}\\
\norm{\op[p]\op[q]-\op[c_{n}]} & \leq C_{d,\GG}^{\prime}\sum_{\substack{\multind m\in\mathbb{N}_{0}^{4d}\\
\abs{\multind m}\leq\GG+3d+3
}
}\int_{\mathbb{R}^{2d}}\brint{\sup_{\substack{\vect x,\vect y\in\mathbb{R}^{2d}\\
\vect x-\vect y=\vect v
}
}\frac{\abs{\partial^{\multind m}F_{n+1}(\vect x,\vect y)}}{\japbrack{\vect x-\vect y}^{\GG}}}\diff{\vect v.}
\end{align*}
The constants $C_{d,\GG}$ and $C_{d,\GG}^{\prime}$ depend only on
$d$ and $\GG$ (not $n$).
\end{lem}
The proof contains ideas used in the usual Weyl calculus adapted for
use in these norm bounds. However, care has been taken to explicitly
express the estimate in terms of the symbol rather than symbol class
seminorms, and to preserve the cancellation between the terms within
$F_{n+1}$. See the remark following the proof for a more detailed
comparison.

\global\long\def\hash{\mathbin{\#}}

\begin{proof}
It suffices to prove the result for $p,q\in\mathcal{S}(\mathbb{R}^{2d})$.
Let $p\hash q$ denote the Weyl symbol of $\op[p]\op[q]$. We have
\citep[(2.44b)]{folland1989harmonic}
\[
p\hash q(\vect z)=\frac{1}{\upi^{2d}}\int_{\mathbb{R}^{2d}}\int_{\mathbb{R}^{2d}}p(\vect z-\vect x)q(\vect z-\vect y)\,\e^{2\i\sigma(\vect x,\vect y)}\diff{\vect y}\diff{\vect x,}\qquad\sigma(\vect x,\vect y)\coloneqq\vect x_{1}\cdot\vect y_{2}-\vect y_{1}\cdot\vect x_{2},
\]
sometimes called the \emph{twisted product} or \emph{Moyal product}
of $p$ and $q$. We apply Taylor's theorem to $p$. The corresponding
term of $p\hash q(\vect z)$ is
\[
T_{j}(\vect z)=\frac{1}{\upi^{2d}}\int_{\mathbb{R}^{2d}}\int_{\mathbb{R}^{2d}}\frac{1}{j!}(-\vect x\cdot\nabla_{p})^{j}(p(\vect z)q(\vect z-\vect y))\,\e^{2\i\sigma(\vect x,\vect y)}\diff{\vect y}\diff{\vect x,}
\]
where $\nabla_{p}$ indicates that the gradient is being taken only
of $p$. Denote $\widetilde{\nabla}_{\vect y}\coloneqq(\nabla_{\vect y_{2}},-\nabla_{\vect y_{1}})$,
so that $2\i\vect x\e^{2\i\sigma(\vect x,\vect y)}=\widetilde{\nabla}_{\vect y}\e^{2\i\sigma(\vect x,\vect y)}$;
then integrating by parts gives 
\begin{align*}
T_{j}(\vect z) & =\frac{1}{\upi^{2d}}\int_{\mathbb{R}^{2d}}\int_{\mathbb{R}^{2d}}\frac{\i^{j}}{j!2^{j}}(\nabla_{p}\cdot\widetilde{\nabla}_{q})^{j}(p(\vect z)q(\vect z-\vect y))\,\e^{2\i\sigma(\vect x,\vect y)}\diff{\vect y}\diff{\vect x}\\
 & =\frac{1}{\upi^{2d}}\int_{\mathbb{R}^{2d}}\int_{\mathbb{R}^{2d}}F_{j}(\vect z,\vect z-\vect y)\,\e^{2\i\sigma(\vect x,\vect y)}\diff{\vect y}\diff{\vect x.}
\end{align*}
By the Fourier inversion theorem this equals $F_{j}(\vect z,\vect z)$.

Denote the remainder term by $R_{n+1}(\vect z)\coloneqq p\hash q(\vect z)-c_{n}(\vect z)$,
which satisfies
\[
R_{n+1}(\vect z)=\frac{1}{\upi^{d}}\int_{\mathbb{R}^{2d}}\int_{\mathbb{R}^{2d}}\int_{0}^{1}(1-t)^{n}\frac{1}{n!}(-\vect x\cdot\nabla_{p})^{n+1}(p(\vect z-t\vect x)q(\vect z-\vect y))\,\e^{2\i\sigma(\vect x,\vect y)}\diff t\diff{\vect y}\diff{\vect x.}
\]
Integrating by parts in the same way as the other terms, we find that
\[
R_{n+1}(\vect z)=\frac{1}{\upi^{2d}}\int_{0}^{1}\int_{\mathbb{R}^{2d}}\int_{\mathbb{R}^{2d}}(n+1)(1-t)^{n}F_{n+1}(\vect z-t\vect x,\vect z-\vect y)\,\e^{2\i\sigma(\vect x,\vect y)}\diff{\vect y}\diff{\vect x}\diff{t.}
\]
Change variables $t\vect x=t\vect u+\half\vect v$ and $\vect y=t\vect u-\half\vect v$.
This has Jacobian $1$ and satisfies $\sigma(\vect x,\vect y)=\sigma(\vect v,\vect u)$,
so 
\[
R_{n+1}(\vect z)=\frac{n+1}{\upi^{2d}}\int_{0}^{1}\int_{\mathbb{R}^{2d}}\int_{\mathbb{R}^{2d}}(1-t)^{n}F_{n+1}\bigl(\vect z-(t\vect u+\half\vect v),\vect z-(t\vect u-\half\vect v)\bigr)\,\e^{2\i\sigma(\vect v,\vect u)}\diff{\vect u}\diff{\vect v}\diff{t.}
\]
Define the operator
\[
P_{\vect x,\vect y}\coloneqq\frac{1-\half\i\vect y\cdot\widetilde{\nabla}_{\vect x}}{1+\abs{\vect y}^{2}}\qquad\implies\transp{P_{\vect x,\vect y}}=\frac{1+\half\i\vect y\cdot\widetilde{\nabla}_{\vect x}}{1+\abs{\vect y}^{2}},
\]
so that $P_{\vect v,\vect u}\e^{2\i\sigma(\vect v,\vect u)}=\e^{2\i\sigma(\vect v,\vect u)}$
and $\transp{P_{\vect u,\vect v}}\e^{2\i\sigma(\vect v,\vect u)}=\e^{2\i\sigma(\vect v,\vect u)}$.
Thus $R_{n+1}(\vect z)$ equals
\[
\frac{n+1}{\upi^{2d}}\int_{0}^{1}\int_{\mathbb{R}^{2d}}\int_{\mathbb{R}^{2d}}(1-t)^{n}\Bigl((P_{\vect u,\vect v})^{M}(\transp{P_{\vect v,\vect u}})^{L}F_{n+1}\bigl(\vect z-(t\vect u+\half\vect v),\vect z-(t\vect u-\half\vect v)\bigr)\Bigr)\,\e^{2\i\sigma(\vect v,\vect u)}\diff{\vect u}\diff{\vect v}\diff{t.}
\]
For the interactions between $P_{\vect u,\vect v}$ and $\transp{P_{\vect v,\vect u}}$
we use the fact that for all $\abs{\multind r}\leq M$ and $\abs{\multind s}\leq L$
we have
\[
\biggl|\partial_{\vect u}^{\multind r}\Biggl(\frac{\vect u^{\multind s}}{\japbrack{\vect u}^{2L}}\Biggr)\biggr|\leq C_{L,M}\frac{1}{\japbrack{\vect u}^{L}},
\]
for some constant $C_{L,M}>0$; this shows that $\abs{\partial^{\vect k}R_{n+1}(\vect z)}$
is bounded by a constant multiple of 
\[
(n+1)\sum_{\abs{\multind l}\leq L}\sum_{\abs{\multind m}\leq M}\int_{0}^{1}\int_{\mathbb{R}^{2d}}\int_{\mathbb{R}^{2d}}(1-t)^{n}\frac{\bigl\vert\partial_{\vect z}^{\multind k}\partial_{\vect u}^{\vect l}\partial_{\vect v}^{\multind m}F_{n+1}\bigl(\vect z-(t\vect u+\half\vect v),\vect z-(t\vect u-\half\vect v)\bigr)\bigr\vert}{\japbrack{\vect u}^{L}\japbrack{\vect v}^{M}}\diff{\vect u}\diff{\vect v}\diff{t.}
\]
Now choose $L=2d+1$, $M=\GG$ and use the trace norm and operator
norm bounds for Weyl operators (see \prettyref{sub:Preliminaries}).
Translating $\vect z'\coloneqq\vect z-t\vect u$ and evaluating the
$\diff t$ integral (which cancels with $n+1$) gives the stated result.
\end{proof}

\begin{rem}
We compare the above lemma with the usual symbolic calculus for Weyl
operators. The decay in the integrand of $R_{n}$ could have been
obtained without changing variables from $(\vect x,\vect y)$ to $(\vect u,\vect v)$,
which bounds $\abs{\partial^{\multind k}R_{n}(\vect z)}$ by a constant
multiple of 
\[
(n+1)\sum_{\abs{\multind l}\leq L}\sum_{\abs{\multind m}\leq M}\int_{0}^{1}\int_{\mathbb{R}^{2d}}\int_{\mathbb{R}^{2d}}(1-t)^{n}\frac{\abs{\partial_{\vect z}^{\multind k}\partial_{\vect x}^{\vect l}\partial_{\vect y}^{\multind m}F_{n+1}(\vect z-\sqrt{t}\vect x,\vect z-\sqrt{t}\vect y)}}{\japbrack{\vect x}^{L}\japbrack{\vect y}^{M}}\diff{\vect y}\diff{\vect x}\diff{t.}
\]
Using the notation of \citet[Chapter~2]{folland1989harmonic}, this
is the bound used to show that if $p\in S_{\rho,0}^{m_{1}}$, $q\in S_{\rho,0}^{m_{2}}$
then $p\#q-c_{n}\in S_{\rho,0}^{m_{1}+m_{2}-\rho(n+1)}$ \citep[see][Theorem~(2.49)]{folland1989harmonic},
although it is not computed explicitly there. This only gives the
result when $\delta=0$; the general case could be handled in a similar
way, but derivatives in $\vect x_{1}$ and $\vect x_{2}$ etc.\ would
need to be tracked separately.
\end{rem}

The next lemma, in combination with the previous one, proves \prettyref{lem:composition}
in the special case that $f(t)=t^{2}$, and contains the essential
idea of the general case. To simplify its statement and use, introduce
the notation
\[
\normnew{\GG,D}F\coloneqq\sum_{\substack{\multind m\in\mathbb{N}_{0}^{4d}\\
\abs{\multind m}\leq D
}
}\int_{\mathbb{R}^{2d}}\int_{\mathbb{R}^{2d}}\frac{\abs{\partial^{\multind m}F(\vect x,\vect y)}}{\japbrack{\vect x-\vect y}^{\GG}}\diff{\vect x}\diff{\vect y.}
\]
In particular, the trace norm bound in \prettyref{lem:composition-norm-bounds}
is a constant multiple of $\normnew{\GG,\GG+4d+2}{F_{n}}$.
\begin{lem}
\label{lem:composition-square}Let $\W,b,\Sigma$ satisfy \prettyref{con:symbol-and-domain}
and let the boundary of $\Sigma$ satisfy $\tau(\partial\Sigma)\geq1$
(see \prettyref{sec:appendix-tubular}). Set
\[
F(\vect x,\vect y)\coloneqq(\nabla_{\vect x_{1}}\cdot\nabla_{\vect y_{2}}-\nabla_{\vect x_{2}}\cdot\nabla_{\vect y_{1}})(\W\ast(b\chi_{\Sigma})(\vect x)\W\ast(b\chi_{\Sigma})(\vect y)).
\]
Set $\GG\coloneqq2d+2$, $D\coloneqq6d+4$. Then 
\[
\normnew{\GG,D}F\lesssim\norm{\nabla b}_{L^{\infty}(\Sigma)}(\norm{\nabla b}_{L^{1}(\Sigma)}+\norm b_{L^{1}(\partial\Sigma)})+\frac{1}{\tau(\partial\Sigma)}\norm b_{L^{\infty}(\partial\Sigma)}\norm b_{L^{1}(\partial\Sigma)}.
\]

\end{lem}
(Recall from \prettyref{sub:Preliminaries} that the constant implicit
in $\lesssim$ may depend on $\W$, but not on $b$ or $\Sigma$.)
\begin{proof}
We have
\[
F(\vect x,\vect y)=\sum_{j=1}^{d}\Bigl(\partial_{(\vect x_{1})_{j}}\partial_{(\vect y_{2})_{j}}-\partial_{(\vect x_{2})_{j}}\partial_{(\vect y_{1})_{j}}\Bigr)(\W\ast(b\chi_{\Sigma})(\vect x)\W\ast(b\chi_{\Sigma})(\vect y)).
\]
For $j\in\{1,\dotsc,d\}$ we have 
\[
\partial_{(\vect z_{1})_{j}}\W\ast(b\chi_{\Sigma})(\vect z)=g_{1,j}(\vect z)+h_{1,j}(\vect z),
\]
where
\[
g_{1,j}(\vect z)\coloneqq\int_{\partial\Sigma}\W(\vect z-\vect z')b(\vect z')(\vect n_{1})_{j}(\vect z')\diff{\vect z'},\quad h_{1,j}(\vect z)\coloneqq\W\ast(\chi_{\Sigma}\,\partial_{(\vect z_{1})_{j}}b)(\vect z),
\]
and similarly for $\partial_{(\vect z_{2})_{j}}\W\ast(b\chi_{\Sigma})(\vect z)$.
Thus, using the symmetry and subadditivity of $\normnew{}{\,\cdot\,}$,
we have
\begin{align*}
\normnew{\GG,D}F & \leq\sum_{j=1}^{d}\Bigl(\normnew{\GG,D}{g_{1,j}(\vect x)g_{2,j}(\vect y)-g_{2,j}(\vect x)g_{1.j}(\vect y)}\\
 & \qquad\mathbin{+}2\normnew{\GG,D}{g_{1,j}(\vect x)h_{2,j}(\vect y)}+2\normnew{\GG,D}{g_{2,j}(\vect x)h_{1,j}(\vect y)}+2\normnew{\GG,D}{h_{1,j}(\vect x)h_{2,j}(\vect y)}\Bigr).
\end{align*}
It is easy to check that
\begin{align*}
\normnew{\GG,D}{g_{1,j}(\vect x)h_{2,j}(\vect y)} & \lesssim\norm b_{L^{1}(\partial\Sigma)}\norm{\nabla b}_{L^{\infty}(\Sigma)},\\
\normnew{\GG,D}{g_{2,j}(\vect x)h_{1,j}(\vect y)} & \lesssim\norm b_{L^{1}(\partial\Sigma)}\norm{\nabla b}_{L^{\infty}(\Sigma)},\\
\normnew{\GG,D}{h_{1,j}(\vect x)h_{2,j}(\vect y)} & \lesssim\norm{\nabla b}_{L^{1}(\Sigma)}\norm{\nabla b}_{L^{\infty}(\Sigma)}.
\end{align*}
It remains to bound the first term. First note that 
\[
g_{1,j}(\vect x)g_{2.j}(\vect y)-g_{2,j}(\vect x)g_{1.j}(\vect y)=\int_{\partial\Sigma}\int_{\partial\Sigma}\W(\vect x-\vect x')b(\vect x')\W(\vect y-\vect y')b(\vect y')\vect m(\vect x',\vect y')\diff{\vect x'}\diff{\vect y',}
\]
where for each $\vect x',\vect y'\in\partial\Sigma$ we set 
\begin{align*}
\vect m(\vect x',\vect y') & \coloneqq(\vect n_{1})_{j}(\vect x')(\vect n_{2})_{j}(\vect y')-(\vect n_{2})_{j}(\vect x')(\vect n_{1})_{j}(\vect y')\\
 & \colonsp=\Bigl((\vect n_{1})_{j}(\vect x')-(\vect n_{1})_{j}(\vect y')\Bigr)(\vect n_{2})_{j}(\vect y')+\Bigl((\vect n_{2})_{j}(\vect y')-(\vect n_{2})_{j}(\vect x')\Bigr)(\vect n_{1})_{j}(\vect y').
\end{align*}
Let $\ell(\vect x',\vect y')$ be the line segment connecting $\vect x'$
to $\vect y'$. When $\abs{\vect x'-\vect y'}\leq\tau(\partial\Sigma)/2$
we have $\ell(\vect x',\vect y')\subseteq\tub(\partial\Sigma,\tau(\partial\Sigma)/2)$
so by \prettyref{lem:norm-gradient-bound} (using the extension of
$\vect n$ defined in \prettyref{sub:tubular-basic}) we have
\[
\abs{(\vect n_{1})_{j}(\vect x')-(\vect n_{1})_{j}(\vect y')}\leq\abs{\vect x'-\vect y'}\sup_{\vect z\in\ell(\vect x',\vect y')}\abs{\nabla\vect n(\vect z)}\leq\frac{2\abs{\vect x'-\vect y'}}{\tau(\partial\Sigma)}.
\]
When $\abs{\vect x'-\vect y'}\geq\tau(\partial\Sigma)/2$ we have
\[
\abs{(\vect n_{1})_{j}(\vect x')-(\vect n_{1})_{j}(\vect y')}\leq2\leq\frac{4\abs{\vect x'-\vect y'}}{\tau(\partial\Sigma)}.
\]
Similar bounds hold for $\vect n_{2}$, so
\[
\abs{\vect m(\vect x',\vect y')}\leq\frac{8\abs{\vect x'-\vect y'}}{\tau(\partial\Sigma)}\leq\frac{24\japbrack{\vect x-\vect x'}\japbrack{\vect x-\vect y}\japbrack{\vect y-\vect y'}}{\tau(\partial\Sigma)}.
\]
We also bound (using \prettyref{lem:window-path-integral} with $U(\vect z)\coloneqq\japbrack{\vect z}\partial^{\multind l}\W(\vect z)$
for the $\diff{\vect x'}$ integral)
\begin{align*}
\int_{\partial\Sigma}\japbrack{\vect x-\vect x'}\abs{\partial^{\multind l}\W(\vect x-\vect x')b(\vect x')}\diff{\vect x'} & \lesssim\norm b_{L^{\infty}(\partial\Sigma)},\\
\int_{\mathbb{R}^{2d}}\int_{\partial\Sigma}\japbrack{\vect y-\vect y'}\abs{\partial^{\multind m}\W(\vect y-\vect y')b(\vect y')}\diff{\vect y'}\diff{\vect y} & \lesssim\norm b_{L^{1}(\partial\Sigma)}.
\end{align*}
We therefore obtain
\[
\normnew{\GG,D}{g_{1,j}(\vect x)g_{2.j}(\vect y)-g_{2,j}(\vect x)g_{1.j}(\vect y)}\lesssim\frac{\norm b_{L^{\infty}(\partial\Sigma)}\norm b_{L^{1}(\partial\Sigma)}}{\tau(\partial\Sigma)}.\qedfix
\]

\end{proof}

\begin{proof}[Proof of \prettyref{lem:composition}]
The way we complete this proof for general $f$ depends on the functional
calculus in use; in other words, it depends on which of the two parts
of \prettyref{con:spectral-function} is satisfied. In both cases
we set $q\coloneqq\W\ast(b\chi_{\Sigma})$, $\GG\coloneqq2d+2$, $D\coloneqq6d+4$,
and 
\[
F(\vect x,\vect y)\coloneqq(\nabla_{\vect x_{1}}\cdot\nabla_{\vect y_{2}}-\nabla_{\vect x_{2}}\cdot\nabla_{\vect y_{1}})(q(\vect x)q(\vect y)).
\]

\ProofCaption{\mbox{\prettyref{con:spectral-function}(1).}} For $j\geq2$
by \prettyref{lem:composition-norm-bounds} we have
\begin{align*}
\norm{\opwone[q^{j+1}]-\opwone[q^{j}]\opwone[q]}_{1} & \lesssim\normnew{\GG,D}{(\nabla_{\vect x_{1}}\cdot\nabla_{\vect y_{2}}-\nabla_{\vect x_{2}}\cdot\nabla_{\vect y_{1}})((q(\vect x))^{j}q(\vect y))}\\
 & \leq(j-1)^{D}(c_{d,\W}\norm b_{L^{\infty}(\Sigma)})^{j-1}\normnew{\GG,D}F,
\end{align*}
where $c_{d,\W}$ is a constant. Summing from $j=1$ to $k-1$ and
bounding the operator norm of $\op[q]$ as in \prettyref{sub:Preliminaries}
we obtain
\[
\norm{(\opwone[q])^{k}-\opwone[q^{k}]}_{1}\lesssim\left(1+(k-2)^{D+2}\right)(c_{d,\W}\norm b_{L^{\infty}(\Sigma)})^{k-2}\normnew{\GG,D}F.
\]
Thus
\begin{align*}
\norm{f(\opwone[q])-\opwone[f(q)]}_{1} & \leq\sum_{k=1}^{\infty}\abs{f^{(k)}(0)}\,\norm{(\opwone[q])^{k}-\opwone[q^{k}]}_{1}\\
 & \lesssim\normnew{\GG,D}F\brsum{\sum_{k=2}^{\infty}\frac{\abs{f^{(k)}(0)}}{k!}\left(1+(k-2)^{D+2}\right)(c_{d,\W}\norm b_{L^{\infty}(\Sigma)})^{k-2}},
\end{align*}
which is convergent. Using \prettyref{lem:composition-square} to
bound $\normnew{\GG,D}F$ gives the result.

\ProofCaption{\mbox{\prettyref{con:spectral-function}(2).}} We may
assume that $f$ is compactly supported because only its values on
a compact interval affect the meaning of $f(\op[q])$. It follows
from the properties of the propagator $\e^{\i t\opwone[q]}$ that
\[
\norm{f(\opwone[q])-\opwone[f(q)]}_{1}\leq\frac{1}{\sqrt{2\upi}}\int_{\mathbb{R}}\brint{\int_{[0,t]}\norm{\opwone[\e^{\i sq}]\opwone[q]-\opwone[\e^{\i sq}q]}_{1}\diff s}\abs{\hat{f}(t)}\diff{t.}
\]
(This may be seen by differentiating the operator $\e^{\i s\opwone[q]}\opwone[\e^{\i sq}]$
with respect to $s$ and integrating on $[0,t]$.) But by \prettyref{lem:composition-norm-bounds}
we have
\begin{align*}
\norm{\opwone[\e^{\i sq}]\opwone[q]-\opwone[\e^{\i sq}q]}_{1} & \lesssim\normnew{\GG,D}{(\nabla_{\vect x_{1}}\cdot\nabla_{\vect y_{2}}-\nabla_{\vect x_{2}}\cdot\nabla_{\vect y_{1}})(\e^{\i sq(\vect x)}q(\vect y))}\\
 & =\normnew{\GG,D}{\i s\e^{\i sq(\vect x)}(\nabla_{\vect x_{1}}\cdot\nabla_{\vect y_{2}}-\nabla_{\vect x_{2}}\cdot\nabla_{\vect y_{1}})(q(\vect x)q(\vect y))}\\
 & \lesssim\japbrack s^{D+1}\japbrack{\norm b_{L^{\infty}(\Sigma)}}^{D}\normnew{\GG,D}F.
\end{align*}
The result then follows from \prettyref{lem:composition-square}.
\end{proof}

\subsection{Step 2: Trace asymptotics\label{sub:trace-asymptotics}}

In this subsection we prove \prettyref{lem:trace-asymptotics}. Set
\begin{align*}
I_{1} & \coloneqq\int_{\mathbb{R}^{2d}}\Bigl(f(\W\ast(\chi_{\Sigma}b)(\vect z))-\W\ast(\chi_{\Sigma}f(b))(\vect z)\Bigr)\diff{\vect z,}\\
I_{5} & \coloneqq\int_{\partial\Sigma}\int_{\mathbb{R}}\Bigl(f(\Qn(\lambda)b(\vect u))-\Qn(\lambda)f(b(\vect u))\Bigr)\diff{\lambda}\mudiff{\vect u}{2d-1}.
\end{align*}
As discussed in \prettyref{sub:proof-overview}, we must show that
when $b=a(\cdot/r)$ and $\Sigma=r\Omega$, we have $I_{1}=I_{5}+O(r^{2d-2})$.
\begin{notation}
In this subsection we will refer to the tubular radius of the boundary
of $\Sigma$ very often, so instead of using the full notation $\tau(\partial\Sigma)$
(which for $\Sigma=r\Omega$ equals $r\tau(\partial\Omega)$) we will
refer to it simply as $\tau$.\end{notation}
\begin{proof}[Proof of \prettyref{lem:trace-asymptotics}]
\ProofCaption{Step 1: Restrict support of $f$.} Depending on which
part of \prettyref{con:spectral-function} is satisfied, either $f$
is a smooth function on $\mathbb{R}$ and $b,\W$ are real-valued,
or $f$ is a smooth function on $\mathbb{C}$. In both cases, $I_{1}$
and $I_{5}$ only depend on the value of $f(t)$ for 
\[
\abs t\leq\norm b_{L^{\infty}(\Sigma)}\int_{\mathbb{R}^{2d}}\abs{\W(\vect z)}\diff{\vect z,}
\]
so we may restrict the support of $f$ to a compact set. In the remainder
of the proof we refer to $\norm f_{L^{\infty}}$ for the supremum
of $\abs f$ over that set, and similarly for $\norm{f'}_{L^{\infty}}$
and $\norm{f''}_{L^{\infty}}$.

\ProofCaption{Step 2: Restrict support of $\W$.} Let $\widetilde{\W}$
be the function defined for each $\vect z\in\mathbb{R}^{2d}$ by
\[
\widetilde{\W}(\vect z)\coloneqq\begin{cases}
\W(\vect z)+K_{\W,\tau} & \text{if }\abs{\vect z}\leq\tfrac{1}{2}\tau,\\
0 & \text{if }\abs{\vect z}>\tfrac{1}{2}\tau,
\end{cases}
\]
with $K_{\W,\tau}$ chosen so that the integral of $\widetilde{\W}$
is $1$. The error in replacing $\W$ by $\widetilde{\W}$ in $I_{1}$
and $I_{5}$ (including the reference to $\W$ in the definition of
$\Q$) is bounded by
\[
2\norm{f'}_{L^{\infty}}\Bigl(\norm b_{L^{1}(\Sigma)}\int_{\mathbb{R}^{2d}}\abs{\W(\vect z)-\widetilde{\W}(\vect z)}\diff{\vect z}+\,\norm b_{L^{1}(\partial\Sigma)}\sup_{\vect{\omega}\in\mathbb{S}^{2d-1}}\int_{\mathbb{R}}\abs{\Qo(\lambda)-\widetilde{\Q}_{\vect{\omega}}(\lambda)}\diff{\lambda}\Bigr).
\]
Since $\W\in\mathcal{S}(\mathbb{R}^{2d})$ these integrals can be
bounded by any negative power of $\tau$; choosing to bound them by
$1/\tau^{2}$ and $1/\tau$ respectively will suffice to satisfy the
required scaling property.

This will be useful later in the proof where certain integrals will
be non-zero outside of a tubular neighbourhood of $\partial\Sigma$
so long as $\W$ has sufficiently small compact support, and this
will allow us to apply the results in \prettyref{sec:appendix-tubular}.
We also have, for each $k\in\mathbb{N}_{0}$, 
\[
\int_{\mathbb{R}^{2d}}(1+\abs{\vect z'})^{k}\abs{\widetilde{\W}(\vect z')}\diff{\vect z'}\leq\int_{\mathbb{R}^{2d}}(1+\abs{\vect z'})^{k}\abs{\W(\vect z')}\diff{\vect z',}
\]
so any bound depending on $\widetilde{\W}$ in this way can be replaced
by one depending on $\W$ uniformly in $\tau$. For the rest of the
proof we use $\widetilde{\W}$ in place of $\W$ without further comment.

\ProofCaption{Step 3: Extract $b$ from convolution.} Let
\[
I_{2}\coloneqq\int_{\mathbb{R}^{2d}}\Bigl(f(\W\ast\chi_{\Sigma}(\vect z)b(\vect z))-\W\ast\chi_{\Sigma}(\vect z)f(b(\vect z))\Bigr)\diff{\vect z.}
\]
We will bound $\abs{I_{1}-I_{2}}$. We can rewrite $I_{1}-I_{2}=\int_{\mathbb{R}^{2d}}(D_{1}(\vect z)-D_{2}(\vect z))\diff{\vect z}$,
where
\begin{align*}
D_{1}(\vect z) & \coloneqq f(\W\ast(\chi_{\Sigma}b)(\vect z))-f(\W\ast\chi_{\Sigma}(\vect z)b(\vect z)),\\
D_{2}(\vect z) & \coloneqq\W\ast(\chi_{\Sigma}f(b))(\vect z)-\W\ast\chi_{\Sigma}(\vect z)f(b(\vect z)).
\end{align*}
Using two terms of the Taylor expansion of $b$ and two terms of the
Taylor expansion of $f$, we obtain
\begin{align*}
 & \int_{\mathbb{R}^{2d}}\abs{D_{1}(\vect z)-(-\vect z'\W(\vect z'))\ast\chi_{\Sigma}(\vect z)\cdot\nabla b(\vect z)f'(\W\ast\chi_{\Sigma}(\vect z)b(\vect z))(\vect z))}\diff{\vect z}\\
 & \qquad\lesssim\norm{f''}_{L^{\infty}}\norm{\nabla b}_{L^{\infty}(\mathbb{R}^{2d})}\norm{\nabla b}_{L^{1}(\mathbb{R}^{2d})}+\norm{f'}_{L^{\infty}}\norm{\nabla(\nabla b)}_{L^{1}(\mathbb{R}^{2d})}.
\end{align*}
 Using two terms of the Taylor expansion of $f(b)$, we obtain
\begin{align*}
 & \int_{\mathbb{R}^{2d}}\abs{D_{2}(\vect z)-(-\vect z'\W(\vect z'))\ast\chi_{\Sigma}(\vect z)\cdot\nabla b(\vect z)f'(b(\vect z))}\diff{\vect z}\\
 & \qquad\lesssim\norm{f''}_{L^{\infty}}\norm{\nabla b}_{L^{\infty}(\mathbb{R}^{2d})}\norm{\nabla b}_{L^{1}(\mathbb{R}^{2d})}+\norm{f'}_{L^{\infty}}\norm{\nabla(\nabla b)}_{L^{1}(\mathbb{R}^{2d})}.
\end{align*}
It thus remains to bound 
\[
\int_{\mathbb{R}^{2d}}\abs{(\vect z'\W(\vect z'))\ast\chi_{\Sigma}(\vect z)\cdot\nabla b(\vect z)(f'(\W\ast\chi_{\Sigma}(\vect z)b(\vect z))-f'(b(\vect z)))}\diff{\vect z.}
\]
This integral is zero outside of $\tub(\partial\Sigma,\tau/2)$. Set
$V(\vect z')\coloneqq(1+\abs{\vect z'})\W(\vect z')$. Therefore, by \prettyref{lem:boundary-value-approx}
and \prettyref{lem:smoothed-compl-cutoffs}, it is bounded by
\begin{align*}
 & \norm{f''}_{L^{\infty}}\int_{\tub(\partial\Sigma,\tau/2)}\abs{(\vect z'\W(\vect z'))\ast\chi_{\Sigma}(\vect z)\cdot\nabla b(\vect z)\W\ast\chi_{\comp{\Sigma}}(\vect z)b(\vect z)}\diff{\vect z}\\
 & \qquad\leq\norm{f''}_{L^{\infty}}\norm b_{L^{\infty}(\mathbb{R}^{2d})}\int_{\tub(\partial\Sigma,\tau/2)}\abs{V\ast\chi_{\Sigma}(\vect z)V\ast\chi_{\comp{\Sigma}}(\vect z)\nabla b(\vect z)}\diff{\vect z}\\
 & \qquad\lesssim\norm{f''}_{L^{\infty}}\norm b_{L^{\infty}(\mathbb{R}^{2d})}(\norm{\nabla b}_{L^{1}(\partial\Sigma)}+\norm{\nabla(\nabla b)}_{L^{1}(\mathbb{R}^{2d})}).
\end{align*}

\ProofCaption{Step 4: Approximate $b$ by its value on $\partial\Sigma$.}
Let 
\[
I_{3}\coloneqq\int_{\tub(\partial\Sigma,\tau/2)}\Bigl(f(\W\ast\chi_{\Sigma}(\vect z)b(\vect u))-\W\ast\chi_{\Sigma}(\vect z)f(b(\vect u))\Bigr)\diff{\vect z,}
\]
where for each $\vect z\in\tub(\partial\Sigma,\tau)$ we define $\vect u\coloneqq\vect z-\delta(\vect z)\vect n(\vect z)\in\partial\Sigma$
(the signed distance function $\delta$ is defined in \prettyref{sub:tubular-basic}).
The integrand of $I_{2}$ is zero outside of $\vect z\in\tub(\partial\Sigma,\tau/2)$,
so by \prettyref{lem:boundary-value-approx} (essentially Taylor's
theorem on $b$ in the $\vect n(\vect u)$ direction) and \prettyref{lem:smoothed-compl-cutoffs}
we have 
\begin{align*}
\abs{I_{2}-I_{3}} & \leq\norm{f''}_{L^{\infty}}\norm b_{L^{\infty}(\mathbb{R}^{2d})}\int_{\tub(\partial\Sigma,\tau/2)}\abs{b(\vect z)-b(\vect u)}\abs{\W\ast\chi_{\Sigma}(\vect z)}\abs{\W\ast\chi_{\comp{\Sigma}}(\vect z)}\diff{\vect z}\\
 & \lesssim\norm{f''}_{L^{\infty}}\norm b_{L^{\infty}(\mathbb{R}^{2d})}(\norm{\nabla b}_{L^{1}(\partial\Sigma)}+\norm{\nabla(\nabla b)}_{L^{1}(\mathbb{R}^{2d})}).
\end{align*}

\ProofCaption{Step 5: Approximate $\Sigma$ locally by a half space.}
Let
\[
I_{4}\coloneqq\int_{\tub(\partial\Sigma,\tau/2)}\Bigl(f(\Q_{\vect n(\vect z)}(\delta(\vect z))b(\vect u))-\Qn(\delta(\vect z))f(b(\vect u))\Bigr)\diff{\vect z,}
\]
where as before we define $\vect u\coloneqq\vect z-\delta(\vect z)\vect n(\vect z)\in\partial\Sigma$.
By \prettyref{lem:tube-change-variables} and \prettyref{lem:det-bounds}
we have
\begin{align*}
\abs{I_{3}-I_{4}} & \lesssim\norm{f'}_{L^{\infty}}\int_{\partial\Sigma}\int_{-\tau/2}^{\tau/2}\abs{b(\vect u)}\abs{\W\ast\chi_{\Sigma}(\vect u+\lambda\vect n(\vect u))-\Qn(\lambda)}\diff{\lambda}\mudiff{\vect u}{2d-1}\\
 & \leq\norm{f'}_{L^{\infty}}\norm b_{L^{1}(\partial\Sigma)}\sup_{\vect u\in\partial\Omega}J(\vect u),
\end{align*}
where for each $\vect u\in\partial\Omega$ we set 
\[
J(\vect u)\coloneqq\int_{-\tau/2}^{\tau/2}\abs{\W\ast\chi_{\Sigma}(\vect u+\lambda\vect n(\vect u))-\Qn(\lambda)}\diff{\lambda.}
\]
We will show that $J(\vect u)\lesssim1/\tau$. We have 
\[
\Qn(\lambda)=\W\ast\chi_{H}(\vect u+\lambda\vect n(\vect u)),\qquad H\coloneqq\{\vect z'\in\mathbb{R}^{2d}:(\vect z'-\vect u)\cdot\vect n(\vect u)\geq0\}.
\]
So, denoting symmetric difference by $\symdiff$, we have 
\begin{align*}
J(\vect u) & \leq\int_{-\tau/2}^{\tau/2}\abs{\W}\ast\chi_{\Sigma\symdiff H}(\vect u+\lambda\vect n(\vect u))\diff{\lambda}\\
 & =\int_{-\tau/2}^{\tau/2}\int_{\Sigma\symdiff H}\abs{\W(\vect u+\lambda\vect n(\vect u)-\vect z')}\diff{\vect z'}\diff{\lambda.}
\end{align*}
This integrand is non-zero only when $\abs{\vect u+\lambda\vect n(\vect u)-\vect z'}<\tau/2$
and $\abs{\lambda}<\tau/2$, so only when $\abs{\vect u-\vect z'}<\tau$.
We may therefore use \prettyref{rem:local-quadratic} with $\vect z'=\vect u+v_{\perp}\vect n(\vect u)+\widetilde{\vect v}$.
This says that $\vect z'\in\Sigma\symdiff H$ only when $\abs{v_{\perp}}\leq\abs{\widetilde{\vect v}}^{2}/\tau$,
so 
\[
J(\vect u)\leq\int_{-\tau/2}^{\tau/2}\int_{\vect n(\vect u)^{\perp}}\int_{-\abs{\widetilde{\vect v}}^{2}/\tau}^{\abs{\widetilde{\vect v}}^{2}/\tau}\abs{\W(\lambda\vect n(\vect u)-v_{\perp}\vect n(\vect u)-\widetilde{\vect v})}\diff{v_{\perp}}\,\mudiff{\widetilde{\vect v}}{2d-1}\diff{\lambda.}
\]
Translating $\lambda$ to $\eta\coloneqq\lambda-v_{\perp}$ and then
setting $\vect x\coloneqq\eta\vect n(\vect u)-\widetilde{\vect v}$,
we obtain 
\begin{align*}
J(\vect u) & \leq\int_{\vect n(\vect u)^{\perp}}\int_{-\abs{\widetilde{\vect v}}^{2}/\tau}^{\abs{\widetilde{\vect v}}^{2}/\tau}\int_{\mathbb{R}}\abs{\W(\eta\vect n(\vect u)-\widetilde{\vect v})}\diff{\eta}\diff{v_{\perp}}\,\mudiff{\widetilde{\vect v}}{2d-1}\\
 & \leq\frac{2}{\tau}\int_{\mathbb{R}^{2d}}\abs{\vect x}^{2}\abs{\W(\vect x)}\diff{\vect x}\lesssim\frac{1}{\tau}.
\end{align*}

\ProofCaption{Step 6: Neglect Jacobian.} By \prettyref{lem:tube-change-variables}
we have
\[
I_{4}=\int_{\partial\Sigma}\int_{-\tau/2}^{\tau/2}\Bigl(f(\Q_{\vect n(\vect z)}(\lambda)b(\vect u))-\Qn(\lambda)f(b(\vect u))\Bigr)\det(I-\lambda S_{\vect u})\diff{\lambda}\mudiff{\vect u}{2d-1}.
\]
In $I_{5}$ the integrand is zero except for when $-\tau/2<\lambda<\tau/2$,
so using \prettyref{lem:det-bounds} to replace $\det(I-\lambda S_{\vect u})$
with $1$, we have 
\begin{align*}
\abs{I_{4}-I_{5}} & \lesssim\frac{1}{\tau}\int_{\partial\Sigma}\int_{-\tau/2}^{\tau/2}\abs{\lambda}\absb{f(\Q_{\vect n(\vect z)}(\lambda)b(\vect u))-\Qn(\lambda)f(b(\vect u))}\diff{\lambda}\mudiff{\vect u}{2d-1}\\
 & \leq\frac{2}{\tau}\norm{f'}_{L^{\infty}}\int_{\partial\Sigma}\abs{b(\vect u)}\mudiff{\vect u}{2d-1}\int_{\mathbb{R}}\abs{\lambda}\absb{\Q_{\vect n(\vect z)}(\lambda)-\chi_{[0,\infty)}(\lambda)}\diff{\lambda}\\
 & \lesssim\frac{1}{\tau}\norm{f'}_{L^{\infty}}\norm b_{L^{1}(\partial\Sigma)}.\qedfix
\end{align*}

\end{proof}

\section{Appendix: Tubular neighbourhood properties\label{sec:appendix-tubular}}

\subsection{Definition and properties\label{sub:tubular-basic}}

Here we recall the definition of tubular neighbourhoods and some of
their basic properties. Throughout this subsection let $\Omega\subseteq\mathbb{R}^{m}$
be a closed set with $C^{2}$ boundary. In practice we will only need
the results when $m$ is even, but everything applies equally to odd
$m$. The material below is well known; see for example \citet[Appendix; moved to \S14.6 in 1983 second edition]{gilbargtrudinger1977}
or \citet{graytubes}.

\begin{notation}
Denote the \textbf{inward} normal vector field by $\vect n\colon\partial\Omega\to\mathbb{R}^{m}$.\end{notation}
\begin{defn}
Let $t>0$. Define the open line segments 
\[
\ell_{\mathrm{nor}}(\vect u,t,\partial\Omega)\coloneqq\{\vect u+\lambda\vect n(\vect u)\in\mathbb{R}^{m}:\lambda\in(-t,t)\}
\]
and define the set
\[
\tub(\partial\Omega,t)\coloneqq\bigcup_{\vect u\in\partial\Omega}\ell_{\mathrm{nor}}(\vect u,t,\partial\Omega).
\]
When the $\ell_{\mathrm{nor}}(\vect u,t,\partial\Omega)$ are disjoint
for all distinct $\vect u\in\partial\Omega$ we call $\tub(\partial\Omega,t)$
a \emph{tubular neighbourhood }of radius $t$.
\end{defn}
For any $t>0$, the set $\tub(\partial\Omega,t)$ is is precisely
the set of points within distance $t$ of $\partial\Omega$. When
$\Omega$ is compact, there always exists a $t>0$ such that $\partial\Omega$
has a tubular neighbourhood of radius $t$; this fact is called the
\emph{tubular neighbourhood theorem}. We denote maximum such radius
that exists by $\tau(\partial\Omega)$ (or set $\tau(\partial\Omega)\coloneqq0$
if no such $t$ exists); it satisfies the scaling relationship, for
$\lambda>0$, 
\[
\tau(\lambda\partial\Omega)=\lambda\tau(\partial\Omega).
\]
When $\tau(\partial\Omega)>0$ we write simply $\tub(\partial\Omega)$
for the tube of this radius; that is,
\[
\tub(\partial\Omega)\coloneqq\tub(\partial\Omega,\tau(\partial\Omega)).
\]

\begin{notation}
For any $\vect z\in\mathbb{R}^{m}$ and $t>0$, we denote the open
ball in $\mathbb{R}^{m}$ centred on $\vect z$ with radius $t$ by
$B(\vect z,t)$.\end{notation}
\begin{rem}
\label{rem:local-quadratic}An equivalent condition to $\partial\Omega$
having a tubular neighbourhood of radius $t$ is that for each $\vect u\in\partial\Omega$
the balls $B(\vect u\pm t\vect n(\vect u),t)$ do not intersect $\partial\Omega$.
This implies that locally the surface $\partial\Omega$ is approximately
flat with uniform quadratic error. To state this explicitly, for $\vect u\in\partial\Omega$,
$\vect z\in\mathbb{R}^{m}$ such that\ $\abs{\vect z-\vect u}\leq\tau(\partial\Omega)$
set $\vect v\coloneqq\vect z-\vect u$, $v_{\perp}\coloneqq\vect v\cdot\vect n(\vect u)$,
$\widetilde{\vect v}\coloneqq\vect v-v_{\perp}\vect n(\vect u)$,
so that $\vect z=\vect u+v_{\perp}\vect n(\vect u)+\widetilde{\vect v}$.
(Then $\widetilde{\vect v}\in\vect n(\vect u)^{\perp}$ i.e.\ $\widetilde{\vect v}$
is in the tangent space at $\vect u$.) Then 
\[
\vect z\in\partial\Omega\quad\Longrightarrow\quad\abs{v_{\perp}}\leq\abs{\widetilde{\vect v}}^{2}/\tau(\partial\Omega).
\]

\end{rem}

\begin{defn}
The \emph{signed distance function} (also called the \emph{oriented
distance function}) is
\[
\delta_{\Omega}(\vect z)\coloneqq\begin{cases}
\hphantom{-}\dist(\vect z,\partial\Omega) & \text{if }\vect z\in\Omega,\\
-\dist(\vect z,\partial\Omega) & \text{if }\vect z\notin\Omega.
\end{cases}
\]
\end{defn}
\begin{lem}
\label{lem:dist-func-properties}Let $\Omega$ have a boundary satisfying
$\tau(\partial\Omega)>0$. Then $\delta_{\Omega}$ is twice continuously
differentiable on $\tub(\partial\Omega)$. Further, let $\vect z\in\tub(\partial\Omega)$,
and set $\vect u\in\partial\Omega$ to the (unique) nearest point
to $\vect z$ in $\partial\Omega$; then
\[
\nabla\delta_{\Omega}(\vect z)=\nabla\delta_{\Omega}(\vect u)=\vect n(\vect u),\qquad\vect z=\vect u+\delta_{\Omega}(\vect z)\vect n(\vect u).
\]

\end{lem}

\prettyref{lem:dist-func-properties} shows that $\nabla\delta_{\Omega}$
is a continuously differentiable extension of the inward normal vector
field, so we write without ambiguity
\[
\vect n(\vect z)\coloneqq\nabla\delta_{\Omega}(\vect z)\qquad\forall\vect z\in\tub(\partial\Omega).
\]
In particular, $\abs{\vect n(\vect z)}=1$ and $(\vect n(\vect z)\cdot\nabla)\vect n(\vect z)=\zerovect$
for all $\vect z\in\tub(\partial\Omega)$.

The primary use of tubular neighbourhoods in this article is to reparametrise
points \uline{near} to $\partial\Omega$ in terms of points \uline{on}
$\partial\Omega$ and the signed distance. To write the Jacobian for
this we need to use the shape operator.
\begin{defn}
For each $\vect u\in\partial\Omega$, define the \emph{shape operator},
also known as the \emph{Weingarten map}, by 
\[
S_{\vect u}\coloneqq\nabla\vect n(\vect u)=\nabla(\nabla\delta_{\Omega})(\vect u).
\]
The associated quadratic form is called the \emph{second fundamental
form}.
\end{defn}
The shape operator is usually defined as $\widetilde{S}_{\vect u}\coloneqq\nabla_{T_{\vect u}\partial\Omega}\vect n(\vect u)$
(the gradient of the normal vector field in the tangent hyperplane),
which is a square matrix of size $m-1$. However, because $(\vect n(\vect u)\cdot\nabla)\vect n(\vect u)=\zerovect$
we have $S_{\vect u}=\widetilde{S}_{\vect u}\oplus0$, so the distinction
will not affect what follows.

Since $S_{\vect u}$ is the Hessian of a real-valued function, it
is a real symmetric matrix, and hence diagonalizable with real eigenvalues
(called the \emph{principal curvature}s). The operator norm of $S_{\vect u}$
equals its (absolutely) largest principal curvature and satisfies
\[
\abs{S_{\vect u}}\leq\frac{1}{\tau(\partial\Omega)}.
\]

\begin{lem}
\label{lem:tube-change-variables}For any $0<t\leq\tau(\partial\Omega)$,
the change of variables
\[
\lambda\coloneqq\delta_{\Omega}(\vect z)\in(-t,t),\quad\vect u\coloneqq\vect z-\delta_{\Omega}(\vect z)\vect n(\vect z)\in\partial\Omega\qquad\Longleftrightarrow\qquad\vect z=\vect u+\lambda\vect n(\vect u)\in\tub(\partial\Omega,t),
\]
has Jacobian $\det(I-\lambda S_{\vect u})$. In other words, for any
$f\in L^{1}(\tub(\partial\Omega,t))$ we have
\[
\int_{\tub(\partial\Omega,t)}f(\vect z)\diff{\vect z}=\int_{\partial\Omega}\int_{(-t,t)}f(\vect u+\lambda\vect n(\vect u))\det(I-\lambda S_{\vect u})\diff{\lambda}\mudiff{\vect u}{m-1}\!.
\]

\end{lem}

We will need one final fact, which will be used to bound the difference
between nearby normals (\prettyref{lem:norm-gradient-bound}).
\begin{lem}
\label{lem:grad-normal-diag}Let $\vect z\in\tub(\partial\Omega)$.
Set $\vect u\in\partial\Omega$ to be the nearest point on $\partial\Omega$
to $\vect z$, and set $U$ to an orthogonal matrix that diagonalises
$S_{\vect u}$ i.e.
\[
S_{\vect u}=U^{-1}\diag\{\kappa_{1},\dotsc,\kappa_{m-1},0\}U
\]
where $\kappa_{j}$ are the principal curvatures at $\vect u$. Then
\[
\nabla\vect n(\vect z)=U^{-1}\diag\left\{ \frac{-\kappa_{1}}{1-\delta_{\Omega}(\vect z)\kappa_{1}},\dotsc,\frac{-\kappa_{m-1}}{1-\delta_{\Omega}(\vect z)\kappa_{m-1}},0\right\} U.
\]

\end{lem}

\subsection{Some basic consequences}

This subsection collects some simple consequences of the tubular neighbourhood
theory described in \prettyref{sub:tubular-basic}, used in \prettyref{sec:Proof}
to prove \prettyref{thm:main}. We will first need a pair of simple
bounds on the Jacobian in \prettyref{lem:tube-change-variables}.
\begin{lem}
\label{lem:det-bounds}For all $\abs{\lambda}\leq\tau(\partial\Omega)/2$
and $\vect u\in\partial\Omega$ we have
\begin{gather*}
\left(\frac{1}{2}\right)^{m-1}\leq\det(I-\lambda S_{\vect u})\leq\left(\frac{3}{2}\right)^{m-1},\\
\abs{\det(I-\lambda S_{\vect u})-1}\leq(2^{m-1}-1)\frac{\abs{\lambda}}{\tau(\partial\Omega)}.
\end{gather*}
\end{lem}
\begin{proof}
These follow immediately by writing $\det(I-\lambda S_{\vect u})$
as the product of $1-\lambda\kappa_{j}$, where $\kappa_{j}$ are
the principal curvatures (in particular, $\abs{\kappa_{j}}\leq1/\tau(\partial\Omega)$
so $\abs{\lambda\kappa_{j}}\leq1$).
\end{proof}

One use of these bounds is the following lemma, which allows Taylor's
theorem in the direction normal to $\partial\Omega$ to be written
with straightforward error terms, rather than using an awkward bound
like
\[
\int_{\partial\Omega}\sup_{\lambda\in(-t,t)}\abs{\nabla a(\vect u+\lambda\vect n(\vect u))}\diff{\vect u.}
\]

\begin{lem}
\label{lem:boundary-value-approx}Let $t\leq\tau(\partial\Omega)/2$,
and let $g$ be a function on $\tub(\partial\Omega,t)$. For each
$\vect z\in\tub(\partial\Omega,t)$ set $\vect u\coloneqq\vect z-\delta_{\Omega}(\vect z)\vect n(\vect z)\in\partial\Omega$.
We have
\begin{align*}
\int_{\tub(\partial\Omega,t)}\abs{a(\vect z)-a(\vect u)}\abs{g(\vect z)}\diff{\vect z} & \leq3^{m-1}\Bigl(\norm{\nabla a}_{L^{1}(\partial\Omega)}+\norm{\nabla(\nabla a)}_{L^{1}(\mathbb{R}^{m})}\Bigr)\sup_{\vect u\in\partial\Omega}\int_{-t}^{t}\abs{\lambda g(\vect u+\lambda\vect n(\vect u))}\diff{\lambda,}\\
\int_{\tub(\partial\Omega,t)}\abs{a(\vect z)-a(\vect u)}\abs{g(\vect z)}\diff{\vect z} & \leq3^{m-1}\norm{\nabla a}_{L^{1}(\mathbb{R}^{m})}\sup_{\vect u\in\partial\Omega}\int_{-t}^{t}\abs{g(\vect u+\lambda\vect n(\vect u))}\diff{\lambda.}
\end{align*}
\end{lem}
\begin{proof}
By \prettyref{lem:tube-change-variables} and \prettyref{lem:det-bounds}
we have
\[
\int_{\tub(\partial\Omega,t)}\abs{a(\vect z)-a(\vect u)}\abs{g(\vect z)}\diff{\vect z}\leq\left(\frac{3}{2}\right)^{m-1}\int_{\partial\Omega}\int_{-t}^{t}\abs{a(\vect u+\lambda\vect n(\vect u))-a(\vect u)}\abs{g(\vect u+\lambda\vect n(\vect u))}\mudiff{\vect u}{m-1}\diff{\lambda.}
\]
Applying Taylor's theorem to $a$ in the normal direction, we find
\[
a(\vect u+\lambda\vect n(\vect u))-a(\vect u)=\lambda\vect n(\vect u)\cdot\nabla a(\vect u)+\int_{0}^{1}(1-s)\lambda^{2}(\vect n(\vect u)\cdot\nabla)^{2}a(\vect u+s\lambda\vect n(\vect u))\diff{s.}
\]
But changing variables $s'=\lambda s$ for $\abs{\lambda}<t$ we have
\[
\absint{\int_{0}^{1}(1-s)\lambda^{2}(\vect n(\vect u)\cdot\nabla)^{2}a(\vect u+s\lambda\vect n(\vect u))\diff s}\leq\int_{-t}^{t}\abs{\lambda(\vect n(\vect u)\cdot\nabla)^{2}a(\vect u+s'\vect n(\vect u))}\diff{s'},
\]
so using \prettyref{lem:tube-change-variables} and \prettyref{lem:det-bounds}
again (this time on $\mudiff{\vect u}{m-1}\diff{s'}$ rather than
$\mudiff{\vect u}{m-1}\diff{\lambda}$) gives the first inequality.

The second inequality follows in exactly the same way, except using
one less term of the Taylor expansion.
\end{proof}

The following two results are used in the composition step (\prettyref{sub:composition}).
\begin{lem}
\label{lem:norm-gradient-bound}For all $\vect z\in\tub(\partial\Omega,\tau(\partial\Omega)/2)$
we have the operator norm bound
\[
\abs{\nabla\vect n(\vect z)}\leq\frac{2}{\tau(\partial\Omega)}.
\]
\end{lem}
\begin{proof}
This follows immediately from \prettyref{lem:grad-normal-diag} using
that each $\abs{\kappa_{j}}\leq1/\tau(\partial\Omega)$.
\end{proof}

\begin{lem}
\label{lem:window-path-integral}Let $\vect z\in\mathbb{R}^{m}$,
$\W\in\mathcal{S}(\mathbb{R}^{m})$. Let $\Omega\subseteq\mathbb{R}^{m}$
have boundary satisfying $\tau(\partial\Omega)\geq1$. Then 
\[
\int_{\partial\Omega}\abs{U(\vect z-\vect u)}\mudiff{\vect u}{m-1}\leq C_{d,U},
\]
where $C_{d,U}$ is a finite constant depending only on $d$ and $U$
(not on $\vect z$ or $\Omega$).\end{lem}
\begin{proof}
Set $\widetilde{U}(\vect u)\coloneqq\sup_{\vect x\in B(\vect u,1/2)}\abs{U(\vect x)}$.
Then the integral is bounded by
\begin{align*}
 & 2^{m-1}\int_{-1/2}^{1/2}\int_{\partial\Omega}\abs{U(\vect z-\vect u)}\det(I-\lambda S_{\vect u})\mudiff{\vect u}{m-1}\diff{\lambda}\\
 & \qquad\leq2^{m-1}\int_{\tub(\partial\Omega,1/2)}\widetilde{U}(\vect z-\vect z')\diff{\vect z'}\leq2^{m-1}\int_{\mathbb{R}^{m}}\widetilde{U}(\vect z')\diff{\vect z'.}\qedfix
\end{align*}

\end{proof}
The following lemma is used in the trace asymptotics (\prettyref{sub:trace-asymptotics})
to show that certain integrands are concentrated close to the boundary
of $\Omega$.
\begin{lem}
\label{lem:smoothed-compl-cutoffs}Let $V\in L^{1}(\mathbb{R}^{m})$,
let $k\in\mathbb{N}_{0}$, and let $t<\tau(\partial\Omega)$. Then
for all $\vect u\in\partial\Omega$ we have 
\[
\int_{-t}^{t}\abs{\lambda^{k}V\ast\chi_{\Omega}(\vect u+\lambda\vect n(\vect u))V\ast\chi_{\comp{\Omega}}(\vect u+\lambda\vect n(\vect u))}\diff{\lambda}\leq\frac{2}{k+1}\int_{\mathbb{R}^{m}}\abs{V(\vect z')}\diff{\vect z'}\int_{\mathbb{R}^{m}}\abs{\vect z'}^{k+1}\abs{V(\vect z')}\diff{\vect z'.}
\]
\end{lem}
\begin{proof}
For $\vect z\in\Omega$ we have $\comp{\Omega}\subseteq\comp{B(\vect z,\dist(\vect z,\partial\Omega))}$,
so
\[
\abs{V\ast\chi_{\Omega}(\vect z)}\leq V_{\mathrm{rad}}(0),\quad\abs{V\ast\chi_{\comp{\Omega}}(\vect z)}\leq V_{\mathrm{rad}}(\dist(\vect z,\partial\Omega)),\quad\text{where }V_{\mathrm{rad}}(\lambda)\coloneqq\int_{\abs{\vect z'}\geq\lambda}\abs{V(\vect z')}\diff{\vect z'.}
\]
Similar relationships hold for $\vect z\in\comp{\Omega}$. But for
$\abs{\lambda}<\tau(\partial\Omega)$ we have $\dist(\vect u+\lambda\vect n(\vect u),\partial\Omega)=\abs{\lambda}$,
so the integral in the lemma statement is bounded by 
\[
2V_{\mathrm{rad}}(0)\int_{0}^{t}\abs{\lambda}^{k}V_{\mathrm{rad}}(\lambda)\diff{\lambda.}
\]
Interchanging the order of integration (between $\diff{\lambda}$
and $\diff{\vect z'}$) gives the result.
\end{proof}

\bibliographystyle{dcu}
\phantomsection\addcontentsline{toc}{section}{\refname}\bibliography{CommonBib}

\end{document}